\documentclass[11pt,oneside]{amsart}
\usepackage{graphicx}
\usepackage{todonotes}

\newcommand{\Id}{\mathrm{Id}}
\newcommand{\tr}{\mathrm{tr}\,}
\def\N{\mathbb{N}}
\def\R{\mathbb{R}}
\def\bsa{{\boldsymbol a}}
\def\bsb{{\boldsymbol b}}
\def\bse{{\boldsymbol e}}
\def\bsw{{\boldsymbol w}}

\def\bsz{{\boldsymbol z}}
\def\bsp{{\boldsymbol p}}
\def\bsq{{\boldsymbol q}}
\def\bsv{{\boldsymbol v}}
\def\bsalpha{{\boldsymbol \alpha}}
\def\bsbeta{{\boldsymbol \beta}}
\def\bseta{{\boldsymbol \eta}}
\def\bsxi{{\boldsymbol \xi}}

\def\dd{\,\mathrm{d}}

\def\tr{\mathrm{tr}\,}
\def\({\left(}
\def\){\right)}

\def\A{F}
\def\di{2}

\newtheorem{theorem}{Theorem}
\newtheorem{definition}[theorem]{Definition}
\newtheorem{lemma}[theorem]{Lemma}
\newtheorem{proposition}[theorem]{Proposition}
\newtheorem{remark}[theorem]{Remark}
\title{Rescaled Objective Solutions of Fokker-Planck and Boltzmann equations}
\author{Karsten Matthies}
\address{Department of Mathematical Sciences, University of Bath,
  Bath BA2 7AY, United Kingdom}
\email{k.matthies@bath.ac.uk}
\author{Florian Theil}
\address{Mathematics Institute, University of Warwick, Coventry CV4 7AL, United Kingdom} \email{f.theil@warwick.ac.uk}
\begin{document}
\begin{abstract}
We study the long-time behavior of symmetric solutions of the nonlinear Boltzmann equation and a closely related nonlinear Fokker-Planck equation. If the symmetry of the solutions corresponds to shear flows, the existence of stationary solutions can be ruled out because the energy is not conserved. After anisotropic rescaling both equations conserve the energy. We show that the rescaled Boltzmann equation does not admit stationary densities of Maxwellian type (exponentially decaying). For the rescaled Fokker-Planck equation we demonstrate that all solutions converge to a Maxwellian in the long-time limit, however the convergence rate is only algebraic, not exponential.
\end{abstract}
\maketitle
\section{Introduction}

Symmetric solutions play a very important role in materials sciences. The reason is that the
fundamental laws of physics exhibit many symmetries such as translation and rotation invariance, those symmetries lead to the existence of time-dependent solutions that are invariant under the action of a symmetry group.

The term \lq objective solution' has been coined    by Dumitric{\u a} and James    in \cite{DJ07} for the case where the symmetry group is a subgroup of the Euclidean symmetry group    motivated by molecular dynamics simulations and other engineering applications.     We will study objective solutions in the case where the symmetries consist of translations. For the purpose of this paper we say that    for a given matrix $S \in \R^{m \times n}$
a function $f:\R^n\to \R$ is $S$-objective if $f(\bsxi + \bseta)= f(\bsxi)$ for all $\bseta \in \mathrm{ker} S$, or equivalently requirement $f(\bsxi) = g(S \, \bsxi)$ for some    $g:\R^m\to \R$.    We will be mostly interested in the kinetic setting where $\bsxi = (\bsz, \bsw)$, $\bsz$ being the position and $\bsw$ the velocity. It is important to realize that translation invariance implies that the configuration space is unbounded, therefore extensive thermodynamic quantities such as energy are automatically
infinite. Moreover as we are dealing with open systems, it is not necessarily the case that local energy densities are conserved even if the equation of motion are conservative.

The properties of the symmetric solutions depend strongly on the choice of $S$,    we analyse here one interesting $S$ which leads to a non-conservative system, but ideas will be also relevant for other $S$.      If $n=2d$, $\Id \in \R^{d\times d}$ is the identity matrix and
$S = (\Id,0)\in \R^{d \times 2d}$ one obtains solutions that are independent of $\bsxi$ and the choice $S = (\Id,\pm \Id)$ yields expanding and contracting flows where $\bsw = \mp \bsz$.
We will study Couette flows/shear flows where
$$
  S = (-\mu\,\bsalpha\otimes\bsbeta,\,\Id),
$$
with $\mu \in \R$ being the shear parameter, $\bsalpha,\bsbeta \in \R^d$ being orthonormal. To see that $S$ corresponds to shear-flows observe that
$$ \mathrm{ker}(S) = \mathrm{span}\{(\bsalpha,0),(\bsbeta,\,\mu\,\bsalpha)\}$$
so that $$f(\bsz+x\,\bsalpha +y\,\bsbeta,\bsw + \mu\,y\,\bsalpha)= f(\bsz,\bsw).$$
One of the key obstacles to studying the long-time behaviour is the fact that stationary solutions do not exist as the energy density of symmetric solutions increases with time. A popular approach to overcome the problem of energy growth is to consider rescaled objective solutions \cite{GS03, DJ10, DJ12, DJ07} and in particular \cite{JNV17}. We revisit the concept of rescaled objective solutions for the Boltzmann equation and a Fokker-Planck equation with similar properties.    In contrast to much of the earlier work,    our results are based on the notion of anisotropically rescaled solutions,    the non-autonomous anisotropic coordinate change will fix the second moment tensor.
We analyze the corresponding rescaled    - now non-autonomous --     equations and obtain the following results for the
nonlinear Fokker-Planck equation (A) and the Boltzmann equation with hard sphere collisions (B).
\begin{itemize}
  \item[A)] Characterization of stationary solutions and sharp estimates of the convergence rate (Theorem~\ref{FP-thm}). The convergence rate is algebraic.
\item[B)] Characterization of the collision invariants and a rigorous proof that stationary solutions are not Maxwellian (Theorem~\ref{Boltz-thm}).
\end{itemize}
The main difference between the nonlinear Fokker-Planck equation and the Boltzmann equation
is that the former has a purely local dissipation term whereas the Boltzmann equation involves a nonlocal
and nonlinear collision operator. As a result we can obtain much more detailed information
about the long-term behaviour of rescaled objective solutions of the Fokker-Planck equation than the Boltzmann equation. In the conservative case it is well known that the Maxwellian is the unique stationary solution of the Fokker-Planck equation and the Boltzmann equation. Moreover solutions of the linear Fokker-Planck equation and the nonlinear, homogeneous Boltzmann equation converge to the equilibrium at an exponential rate, cf. \cite{DV06} and \cite{M06}. For the inhomogeneous Boltzmann equation the problem of establishing
exponential convergence to the equilibrium is closely linked to Cercignani's conjecture, an overview can be found in \cite{DMV11}.

The behaviour of the rescaled objective solutions is quite different.
In the case of the Fokker-Planck equation the equilibrium after the anisotropic scaling is still a Maxwellian, but the rate of convergence is only algebraic.
While it is not known whether the rescaled Boltzmann equation for hard spheres admits stationary solutions
our results imply that even if one exist it is not of exponential type. In particular, Maxwellians are not equilibria. We point out that existence of renormalized stationary solutions of the Boltzmann equation with Maxwellian interaction has been established in \cite{JNV17}.

The main method to analyse the long-term behaviour of the Fokker-Planck equation is an adaption of hypocoercivity in a non-autonomous setting. Convergence to equilibria in degenerate dissipative equations preserving mass
has attracted major interest  starting with the use of logarithmic Sobolev inequalities, entropies  and other tools functional analytic tools \cite{OV00,MV00}. These methods could be applied
to Fokker-Planck equations \cite{AMTU01,CJMTU01} as well as some Boltzmann equations \cite{CCG03,DV05}. A general abstract approach for evolution equations consisting a (possibly)
 degenerate dissipative part and some conservative part was introduced by Villani with his concept of hypocoercivity \cite{VillHyp}, see also \cite{DMS15}. This method has  successfully been adapted in many contexts
 like  a linear operator in some vorticity formulations \cite{GGN}, a wide class of dissipative kinetic equations \cite{Dua11}, a generalized Langevin
equation \cite{OP11} and the  meta-stability of bar states in Navier-Stokes equations \cite{BW13}.
Recent extensions of the theory include \cite{DMS15} for classes of  linear kinetic equations, \cite{MMArma15} for kinetic Fokker-Planck equations, and \cite{Bau17} for a modified general approach using
a generalised Bakry-\'{E}mery calculus.

Our methodological contribution is an adaption to non-autonomous nonlinear equations by combining the abstract hypocoercivity result for a limiting problem in a Duhamel formula with a priori estimates for higher derivatives of the full equation. These a priori estimates are indeed obtained using a calculus inspired by hypocoercivity.     A crucial ingredient is the detailed asymptotic analysis  of the anisotropic rescaling, which can be obtained from closed ordinary differential equations for the second moments of the rescaled Fokker-Planck solutions. Indeed, higher order moment equations are used to derive lower algebraic estimates in the convergence rate for typical initial data.  The lack of detailed knowledge about the second moments implies that we have a less explicit control of the  anisotropic rescaling in case of objective solutions to the Boltzmann equation, such that the characterization of a limit distribution and their convergence rates is beyond the scope of this paper.

The rest of the paper is organised as follows. In Section 2 we collect some fundamental properties of objective functions. The results for Fokker-Planck situation are given and proved  in Section 3. The corresponding analysis for the Boltzmann equation is in Section 4. We give a short summary and conclusion in Section 5. The proofs of some technical results not relevant for the main argument are postponed to the Appendix.

\section{Objective functions}
\begin{definition}
Let $S\in \R^{l \times n}$ a matrix. A function $f\in L^1_\mathrm{loc}(\R^n)$ is called $S$-objective
if $f(\bsxi + \bseta)= f(\bsxi)$ for all $\bseta \in \mathrm{ker} S$.
\end{definition}
A classical result for functions which are invariant under the action of a symmetry group is the
Hilbert-Weyl theorem which states that the ring of invariant polynomials has a basis, cf. e.g. \cite{GSS88}. We require a closely related result for measurable functions.
\begin{proposition} \label{HW}
   Let $S\in \R^{l \times n}$ a matrix.    Let $f\in L^1_\mathrm{loc}(\R^n)$ be a measurable function.
The following are equivalent:
\begin{enumerate}
\item $f$ is $S$-objective.
\item $\nabla\cdot (f T)=0$ if $T \in \R^{n \times l}$ has the property that $\mathrm{range} \,T = \mathrm{ker}S$.
\item There exists a measurable function $g:\mathrm{range}(S) \to \R$ such that $f(\bsxi) = g(S\bsxi)$.
\end{enumerate}
\end{proposition}
The proof is standard, we include it for the convenience of the reader.
\begin{proof}
(1) implies (2):\\
It suffices to show that $\int f\, \nabla \varphi\cdot \bseta \dd \bsxi=0$ for each $\bseta \in \mathrm{ker} S$ and each smooth and compactly supported testfunction $\varphi$.
As $f$ is $S$-objective one finds that
$$0 =\lim_{h\to 0} \frac{1}{h}\int (f(\bsxi+h \bseta)-f(\bsxi))\varphi(\bsxi)\, \dd \bsxi = \lim_{h\to 0} \frac{1}{h}\int (\varphi(\bsxi -h \bseta)-\varphi(\bsxi))\, f(\bsxi)\, \dd \bsxi = - \int (\nabla \varphi \cdot \bseta)\,f\, \dd \xi,$$ which is the claim.\\[1em]
(2) implies (1):\\
As $\mathrm{range} \,T = \mathrm{ker}S$ there exists $\bsa \in
\R^{l}$ such that $\bseta = T\bsa$.
Then
\begin{align*}
f(\bsxi+\eta) - f(\bsxi) = \int_0^1 \frac{\dd}{\dd s} f(\bsxi+s \eta)\, \dd s = \int_0^1 \nabla f(\bsxi+s \eta) \cdot \eta\,\dd s
 = \int_0^1 \nabla \cdot( f(\bsxi+s \eta) T) \bsa\,\dd s=0.
\end{align*}\\[0.1em]
(1) implies (3):\\
Define the operator $\bar S: \mathrm{range}(S^*) \to \mathrm{range}(S)$ by $\bar S= S|_{\mathrm{range}(S^*)}$. Observe that $\bar S$ is invertible and define $g(\bsxi) = f(\bar S^{-1}\bsxi)$.\\[1em]
%\todo[inline]{Unclear, the proof should use the $T$}
(3) implies (1):\\
If $\bseta \in \mathrm{ker} S$, then
$$ f(\bsxi+\bseta)= g(S(\bsxi+\bseta))=g(S\bsxi)= f(\bsxi).$$
\end{proof}
We are interested in a shear flow setting where $\bsalpha, \bsbeta \in \R^d$ are orthonormal
vectors, $n=2d$ and $$S = (-\mu\,\bsalpha\otimes\bsbeta,\,\Id) \in \R^{d\times 2 d}.$$

As $\ker S = \mathrm{span}\{(\bsalpha,0), (\bsbeta,\mu \bsalpha)\}$ any $S$-objective function $f$ satisfies
\begin{equation}
\label{shear}
f(\bsz, \bsw) = f(\bsz +x\,\bsalpha +y\, \bsbeta, \bsw +\mu\,y\,\bsalpha ) \text{ for all }
x,y \in \R.
\end{equation}
Moreover, by Proposition~\ref{HW} part (2)
\begin{eqnarray*}
\nabla_{\bsz} f\cdot \bsalpha &=&0,\\
\nabla_{\bsz} f\cdot \bsbeta+ \mu \nabla_{\bsw} f \cdot \bsalpha&=&0,
\end{eqnarray*}
or equivalently
\begin{equation} \label{symrel}
\nabla_{\bsz} f =- \mu\,(\nabla_{\bsw} f\cdot \bsalpha)\,\bsbeta.
\end{equation}
Our results are based on the observation that the representation of objective functions as in Proposition \ref{HW} is not unique because $S$ is not fully determined by the null space. A careful choice of the representation
can lead to interesting results.
\begin{definition}
   Let $S\in \R^{n \times d }$  be a matrix.       A function $f$ is {\em rescaled} $S$-objective  if it admits the representation
\begin{equation} \label{sym}
f(\bsxi) =\det \eta \; G(\eta \, S\, \bsxi) %\qquad g(x,u-\mu\,y,v) \quad f(\bsz,\bsw) = g(\bsz\cdot \bsalpha, \bsw - \mu \bsalpha \otimes \bsbeta \bsw).
\end{equation}
for some density $G$,
where $\eta \in \R^{d\times d}_\mathrm{sym}$.
\end{definition}
In the shear flow setting one obtains the scaling relation
\begin{equation} \label{renorm}
\bsp = \eta \,(\bsw + \mu \bsalpha \otimes \bsbeta \bsz),
\end{equation}
and the corresponding differential relation
\begin{equation} \label{diff-rel}
\nabla_\bsw f = \eta \nabla_{\bsp} G.
\end{equation}
Rescaled solutions for the Boltzmann equation in shear flow settings have been considered in numerous publications, in particular \cite{DJ10} and \cite{GS03}. Our main contribution to this
topic is the consideration of a renormalization operator $\eta$ which is non-isotropic, i.e. $\eta \neq \lambda\, \Id$ for all $\lambda \in \R$.

\section{The Fokker Planck case} \label{sec:FP}
The Fokker-Planck equation is typically considered as the Kolmogorov forward equation of a Brownian particle in a fluid. It has also been proposed as an approximation of the Boltzmann equation e.g. in \cite{LFH60, CUF70}.
Furthermore \cite{Gou97,FPTT} use Fokker-Planck equations to study grazing collisions in the Boltzmann equation and the Kac model. Carlen and Gangbo use a Fokker-Planck equation also as model problem in \cite{CarGan} for the descent in a Wasserstein metric in kinetic equations, further extensions are given in \cite{CarMaa}.

Normally the kinetic energy $\theta$ is a fixed parameter in the Fokker-Planck equation. In our setting we assume that $\theta$ depends on the density $f$, as a result
the structural properties of the solutions are very similar to the solutions of the Boltzmann equation. In particular mass, momentum and energy are conserved, however energy conservation
only holds for $\mu=0$.
Let $\bsxi = (\bsz, \bsw) \in  \R^{2d}$
%\todo[inline]{Not a matrix $2d$?} with the convention $\bsz =(x,y)$ and $\bsw = (u,v)$. We seek solutions $f_t(\bsxi)$  of the non-linear PDE
%\begin{subequations}
%\label{fp}
\begin{eqnarray}
\label{fp}
\begin{cases} \partial_t f_t(\bsxi) = Lf_t(\bsxi)  & \bsxi \in  \R^{2d}, t>0,\\
f_0(\bsxi)= g_0(S \,\bsxi) & \bsxi \in  \R^{2d}, t=0,
\end{cases}
\end{eqnarray}
%\end{subequations}
with $g_0 \in L^1(\R^d)$, $g_0 \geq0$,

\begin{align*}
 Lf (\bsz,\bsw)&=-\bsw\cdot \nabla_\bsz f(\bsz,\bsw) +\Delta_\bsw f(\bsz,\bsw) +\frac{\rho (\bsz)\,d}{2\theta(\bsz)}\,\nabla_\bsw\cdot(f(\bsz,\bsw)\,(\bsw-\frac{1}{\rho(\bsz)}\bsv(\bsz)),
\end{align*}
and thermodynamic quantities depending on the space variable $\bsz$
\begin{align*}
\rho(\bsz) &= \int f(\bsz,\bsw)\,\dd \bsw && \text{ (density)},\\
\theta(\bsz)&= \frac{1}{2}\int |\bsw-\rho^{-1}\bsv(\bsz)|^2\,f(\bsz,\bsw)\,\dd \bsw && \text{ (kinetic energy)},\\
\bsv(\bsz)&= \int \bsw\, f(\bsz,\bsw)\,\dd \bsw && \text{ (momentum)}.
\end{align*}

For the solutions of interests integration over $\bsz$ will not lead to finite quantities. However
the motivation for \eqref{fp} is that it is  similar to the classical Boltzmann equation as it has comparable conservation properties. To see this we define    for $S$-objective solutions    the standard thermodynamic quantities,    which can depend on time along a solution $f_t$,  by evaluating at $\bsz=0$
\[ m_t = \rho_t(0) \text{(mass),} \quad  \bar \bsv_t =\bsv_t(0) \text{(momentum) and } \theta_t = \theta_t(0)  \text{ (energy).} \]
   The values for other $\bsz$ are then determined by objectivity.
\begin{proposition} \label{fpcons}
   Let $d=2$  and let $f$ be a solution of \eqref{fp} and \eqref{shear} such that $\sup_{0\leq t< T}\theta[f_t]<\infty$.
Then there exists $g_t$ such that $f_t(\bsxi) = g_t(S\,\bsxi)$. Furthermore, mass $m$ and $\bar \bsv$ are conserved. %\cdot \bsbeta$

If $\mu=0$, then energy $\theta$ is also conserved. If
$f^M$ is a Maxwellian, i.e. $f^M(\bsw) = \exp(h(\bsw))$ and
$$ h(\bsw)= a+ \bsb \cdot \bsw + c\,|\bsw|^2,$$
for some $a\in \R$, $\bsb \in \R^\di$, $c<0$, then $f$ is a stationary solution. Any spatially homogenous $f $ with $f(.)(1+|.|^2) \in L^1(\R^\di)$ converges to some $f^M$ with an exponential rate
as $t \to \infty$.
\end{proposition}

   The existence of $g_t$ immediately follows from Proposition \ref{HW}. The rest of    the proof     mainly involves direct calculations, which we postpone  to    the appendix.

Equations \eqref{fp}    and     \eqref{shear} do not admit stationary solutions if $\mu \neq 0$.
We now aim to characterize the asymptotic behavior of objective solutions for non-zero $\mu$.
The main result of this section states that there exists a time-dependent rescaling operator $\eta_t$ such that $G_t$ converges to a Maxwellian as $t\to \infty$.

\begin{theorem} \label{FP-thm}
Let $d=2$. There exists $\eta_t \in C^1([0,\infty),\R^{\di \times \di}_\mathrm{sym})$ such that the rescaled
Fokker-Planck equation
\begin{equation}\label{eq:shape} \left\{
\begin{array}{rll}
\partial_t G_t&=\nabla\cdot\(G_t(\bsp)\,(\theta{_t}^{-1} \Id - \A{_t})\bsp + \eta_t^2\nabla G_t\), \quad & t>0, \, \bsp \in \R^\di,\\[0.5em]
\A{_t}&= \(\dot \eta_t -\mu\, \eta_t\, \bsalpha \otimes \bsbeta\)\eta_t^{-1}, & t>0,
\end{array}\right.
\end{equation}
admits a global solution $G_t$ if $G_0 \in L^1\cap L^\infty$ and $\int_{\R^\di} G_0(\bsp) (1+|\bsp|^2)\, \dd \bsp < \infty$.  The density $f$, which is defined by \eqref{sym}, satisfies \eqref{fp}    and \eqref{shear}.

Furthermore, assume
$\int_{\R^\di} G_0(\bsp)\,\dd \bsp=1$ and  $\int_{\R^\di} G_0(\bsp) \bsp\, \dd \bsp=0$.
The density $G_t$ converges to the Maxwellian $G^M(\bsp)= (4\pi)^{-\frac{d}{2}}\exp(-\frac{1}{4} |\bsp|^2)$ for large $t$ in the $L^1$ sense with an algebraic rate, i.e. there exist $\lambda_-,\lambda_+ >0$ such that
\begin{align}
\label{eq:Rate} \limsup_{t \to \infty} t^{\lambda_-} \| G_t-G^M\|_{L^1(\R^\di)} <\infty & \mbox{ for all } G_0
\end{align}
and
\begin{align}
\label{eq:ratelower}\liminf_{t \to \infty} t^{\lambda_+} \| G_t-G^M\|_{L^1(\R^\di)}>0
\end{align}
for     $G_0$  in an open dense set of admissible initial data with   $\int_{\R^\di} G_0(\bsp) (1+|\bsp|^6)\, \dd \bsp < \infty$.

Furthermore, for $t \to \infty$ the rescaling operator $\eta_t$ admits the asymptotics:
\begin{eqnarray} \label{etaassymp}
\eta_t  &=& \frac{1}{\mu\, t^{\frac{3}{2}}+O(t)}(\sqrt{3} \bsalpha\otimes \bsalpha + 3(\bsalpha \otimes \bsbeta + \bsbeta \otimes \bsalpha) + 2\mu t \bsbeta \otimes \bsbeta).
\end{eqnarray}
\end{theorem}

\begin{remark} \label{FP-ent}
\begin{enumerate}
\item The assumption that $d=2$ is not necessary. The
same result can be obtained if $d\geq 2$ at the expense of
more complicated notation.
\item If $\mu\neq 0$ the energy is not conserved. As a result \eqref{fp} is nonlinear, hence even long-time existence and uniqueness of solutions is not completely trivial.
\item
The fact that Maxwellians are global attractors of the dynamics is typically attributed to the
observation that the entropy is a Lyapunov functional. We show in the appendix that the functional
\begin{equation}\label{eq:lyp}
    S[G]= \int_{\R^\di} \left(\log G   (\bsp)    + \frac{1}{2}|\bsp|^2\right)\, G(\bsp)\,\dd \bsp
\end{equation}
decreases for solutions of Fokker-Planck equation under shear    $S$-objectivity.      However, this observation is not sufficient for the solutions to converge to the minimum of $S$ as the dissipation operator $\nabla\cdot ( \eta_t^2 \nabla .)$ degenerates for $t \to \infty$ as in \eqref{etaassymp}.
% becomes very singular as our proof of Theorem~\ref{FP-thm} shows.
\item
The algebraic order $\lambda_->0$ follows from Proposition~\ref{prop:contract}    below. It    is not explicit as we use an abstract result of \cite{VillHyp} to obtain it.
Similarly,    our calculation of    the constant for the lower bound $\lambda_+$ is relatively crude.    The lower algebraic estimates are based on a detailed understanding of fourth and sixth order moment equations. The analysis  provides lower estimates for all such data, which have -after rescaling with $\eta_t$- different fourth moments compared to $G^M$. However, our method does not provide explicit estimates if the fourth moments of the initial distribution and the corresponding Maxwellian coincide.
\item    We can also consider general $\int_{\R^\di} G_0(\bsp)\,\dd \bsp=m>0 $ and  $\int_{\R^\di} G_0(\bsp) \bsp\, \dd \bsp=m \bsv \in \R^2$.  A translation of the coordinate system can remove the drift, the different mass will need to be introduced in the normalisation condition \eqref{covG} for $\eta_t$ below. Then $G$ will converge $m G^M(. +\bsv)$.

\end{enumerate}
\end{remark}

The proof will take up the rest of this section.    It involves several steps.
\begin{enumerate}
  \item In the beginning of subsection \ref{subsec1} we derive a differential equation for the representative $g$ of the $S$-objective function $f$ and construct a solution to this equation in Proposition \ref{prop:solFP}.
  \item In subsection \ref{sub:moment} we define the rescaling operator $\eta_t$ and the shape $G$. Their asymptotic behaviour is obtained from a closed system of moment equations as stated in Proposition \ref{prop:star}.
  \item The main ingredients of the convergence proof is given in subsection \ref{sec:shape}. We show that the Maxwellian $G^M$ is an equilibrium and use Proposition \ref{prop:star} to identify the leading terms
  in \eqref{eq:shape}. After an appropriate rescaling of time the equation has the form of a autonomous degenerate parabolic part plus  small non-autonomous perturbations. Hypocoercivity estimates are used for the autonomous degenerate parabolic part in $H^1$ relative to the Maxwellian. Additional a priori estimates for the full equation in  higher Sobolev norms are provided using calculations inspired by the hypocoercivity framework. The convergence results follows with a Duhamel formula.
  \item The equations for fourth and sixth moments are used in \ref{subsec:lower} to obtain the lower estimates.
  \item The proof is summarised in subsection \ref{subsec:final}.
\end{enumerate}

\subsection{Reformulation and regularity}\label{subsec1}
To minimize the notation we will assume that $\int G_0(\bsp)\, \dd \bsp = 1$ and
$\int G_0(\bsp)\, \bsp \,\dd \bsp =0$, i.e. $m=1$ and $\bar \bsv=0$.

If $f_t$ is an objective solution, i.e. $f_t(\bsz,\bsw) = g_t(\bsw + \mu \bsalpha \otimes \bsbeta \bsz)$ then by \eqref{fp} and \eqref{symrel} $g_t$ satisfies
\begin{eqnarray} \nonumber
\partial_t g &=& \mu\,(\nabla_\bsw g\cdot \bsalpha)\,(\bsbeta\cdot\bsw) +\Delta_\bsw g+ \theta^{-1}\,\nabla_\bsw\cdot(g(\bsw)\,\bsw)\\
\label{sred}
&=& \nabla_\bsw\cdot\(g(\bsw)\(\theta^{-1} \Id + \mu \bsalpha \otimes \bsbeta\)\bsw\)+\Delta_\bsw g.
\end{eqnarray}
A rescaled objective solution $G_t(\bsp)= \det \eta_t^{-1}\,g_t(\eta_t^{-1}\bsp)$ with $\bsp = \eta_t \bsw$ satisfies
\begin{eqnarray*}
\partial_t G_t(\bsp) &=& \partial_t \(g(\eta_t^{-1}\bsp)\,\det \eta_t^{-1}\)\\
&=& (\partial_t g - \nabla_\bsw g \cdot\,\eta_t^{-1}\dot \eta_t\eta_t^{-1} \,\bsp- \tr(\eta_t^{-1}\dot \eta_t)\,g)\,\det \eta_t^{-1}\\
&=&\(\partial_t g - \nabla_\bsw\cdot (g\, \eta_t^{-1}\dot \eta_t \, \bsw)\)\det \eta_t^{-1}\\
&=& \(\nabla_\bsw\cdot\(g(\eta_t^{-1}\bsp)\(\theta^{-1} \Id + \mu \bsalpha \otimes \bsbeta\)\bsw\)+\Delta_\bsw g - \nabla_\bsw\cdot(g\, \eta_t^{-1}\dot \eta_t \, \bsw)\)\,\det \eta_t^{-1}
\end{eqnarray*}
Now observe that $\nabla_\bsw = \eta_t \nabla_\bsp$. Continuing the above calculation we obtain
\begin{eqnarray*}
\partial_t G_t &=& \(\nabla_\bsp\cdot\eta_t \(G_t(\bsp)\,\(\theta^{-1} \Id + \mu \bsalpha \otimes \bsbeta\)\eta_t^{-1}\bsp\) + \nabla_\bsp(g_t\, \dot \eta_t \eta_t^{-1} \bsp)\)\,\det \eta_t^{-1} +\nabla_\bsp \cdot \eta_t^2 \nabla_\bsp G_t\\
&=&\nabla_\bsp\cdot\(G_t(\bsp)\,(\theta^{-1} \Id - \(\dot \eta_t -\mu\, \eta_t\, \bsalpha \otimes \bsbeta\)\eta_t^{-1})\bsp + \eta_t^2\nabla_\bsp G_t(\bsp)\)
\end{eqnarray*}
which is \eqref{eq:shape}.

Next we show that equation \eqref{eq:shape} admits unique solutions for arbitrary times. It suffices to consider the case $\eta_t = \Id$ because for a general function $\eta_t \in C^1([0,\infty), \R^{\di\times \di})$ the density $G(\bsp) = \det \eta_t^{-1}\, g(\eta_t^{-1} \bsp)$ satisfies \eqref{eq:shape} if $g$ solves \eqref{sred}.

We formulate the underlying regularity result next. The diffusion term $\Delta_\bsw g$ is the generator of the strongly continuous semigroup on $L^2(\R^\di)$ via convolution with the classical
heat kernel $\Phi_t(\cdot)=\frac{1}{4\pi t} \exp\(-\frac{|\cdot|^2}{4t}\)$
for $t>0$. Following e.g. \cite{Ku},  the  kernel $\Phi_t(\cdot)$  also generates equispectral semigroups
 on weighted $L^p$ spaces like $ L^2_2(\R^\di)$, i.e.  the space of integrable function that satisfy
$\int_{\R^\di}g^2_0(\bsw) (1+|\bsw|^2)\, \dd \bsw < \infty$.

Due to $\theta$ the equation \eqref{sred} is nonlinear in $g$, furthermore the factor $w$ makes the divergence terms unbounded. Hence we need to take care to define  a mild solution to
 \eqref{sred} to be a solution in $ L^2_2(\R^\di)$ of the form
\begin{equation}\label{eq:fokkmild}
 g_t = [\Phi_t \ast g_0] + \int_0^t \Phi_{t-s}(.) \ast( \mu\,\nabla\cdot(g_s(.)\,\bsalpha \otimes \bsbeta . )+ \theta^{-1}\,\nabla\cdot(g_s(.) \, .)  \dd s.
\end{equation}

\begin{proposition}\label{prop:solFP} Let $g_0 \in L^1(\R^\di) \cap L^\infty(\R^\di)$.
If $\int_{\R^\di} g_0(\bsw) (1+|\bsw|^2)\, \dd \bsw < \infty$,
then \eqref{sred} admits a unique    mild    solution for all $t>0$ such that $\int_{\R^\di} g_t(\bsw) (1+|\bsw|^2)\, \dd \bsw < \infty$. Furthermore $g_t$ is smooth for $t>0$.
\end{proposition}
\begin{proof}
By H\"older's inequality, we have $ g_0 \in L^2_2(\R^\di)$, {}   such that the first term $\Phi_t \ast g_0$ in  \eqref{eq:fokkmild} is  well-defined.  {} We will obtain    a mild solution as in \eqref{eq:fokkmild}  via an approximation scheme using non-autonomous bounded perturbations.   {} Let
\[\chi_n(\bsw)=\begin{cases}
 \bsw, & \mbox{if } |\bsw|<n \\
      n \bsw /|\bsw|, & \mbox{otherwise}.
    \end{cases}
\]
and $\theta_0=1$, we define recursively for $n \in \N$ as  a non-autonomous Miyardera perturbation,  see e.g.  \cite{RRSV}, the following mild solution
\begin{equation}\label{eq:fokkmildn}
 g^n_t = [\Phi_t \ast g_0] + \int_0^t \Phi_{t-s}(.) \ast( \mu\,\nabla \cdot( g_s(.)\,\bsalpha \otimes \bsbeta \chi_n(.)) +\theta_{n-1}^{-1}\,\nabla\cdot(g_s(.) \, \chi_n(.))  \dd s.
\end{equation}
By the properties of the convolution, we see that $g^n$ is smooth for $t>0$
 and gives a classical    solution    as the second convolution in
\eqref{eq:fokkmildn} is well-defined. For a fixed time $T>0$ standard a priori estimates give uniform bounds in $L^\infty((0,T),  L^2_2(\R^\di))$ and $L^2((0,T),H^1(\R^\di))$.
Differentiating $\theta_n$ with respect to $t$ gives
\begin{eqnarray*}
\frac{\dd \theta_n}{\dd t} &=& \frac{1}{2}\int_{\R^\di} |\bsw|^2( \mu\,\nabla\cdot( g_t\,\bsalpha \otimes \bsbeta \bsw) +\Delta g_t+ \,\theta_{n-1}^{-1}\,\nabla\cdot(g_t\,\bsw))\,\dd \bsw\\
&=&  \underbrace{\frac{1}{2}\int_{\R^\di} |\bsw|^2 \Delta g_t\, \dd \bsw}_{=\di}
-\mu \int_{\R^\di} (\bsw \cdot\bsalpha) (\bsbeta \cdot\bsw) \,g_t\, \dd \bsw- \,\theta_{n-1}^{-1}\underbrace{\int_{\R^\di} |\bsw|^2 \,g_t\, \dd \bsw}_{=2\theta_n} \\
&\leq&  C+ (|\mu| + \frac{2}{\theta_{n-1}} )\, \theta_n .
\end{eqnarray*}
Thus with a Gronwall estimate, $\theta_n$ and $\dot \theta_n$ remain bounded. Similarly
\begin{eqnarray*}
\frac{\dd \theta_n}{\dd t}  (\theta_n)^{-1} &=& \frac{-1}{\theta_n^2} (2 -\mu \int_{\R^\di} (\bsw \cdot\bsalpha) (\bsbeta \cdot\bsw) \,g_t\, \dd \bsw -  \theta_{n-1}^{-1} 2\theta_n \\
&\leq&  \frac{-2}{\theta_n^2}   + (|\mu| + \frac{2}{\theta_{n-1}}) \, (\theta_n)^{-1},
\end{eqnarray*}
which also shows that $\theta_n^{-1}$ and $\frac{\dd \theta_n}{\dd t}  (\theta_n)^{-1} $  remain bounded on $(0,T)$.

All bounds combined give a subsequence such that  $g_n \to g$ weakly  in  $L^2((0,T),H^1(\R^\di))$,  $g_n \to g$ weak  star  in $L^\infty((0,T),  L^2_2(\R^\di))$, $\theta_n \to \theta$ strongly in $C^0(0,T)$, such that the nonlinear term will converge weakly to $\theta^{-1}\,\nabla_\bsw\cdot(g(.)\, .)$. Standard arguments as in    \cite[Chap.2]{Tem}    give then that $g$ satisfies \eqref{eq:fokkmild}. To show uniqueness consider two solutions $g,h$ and a priori estimates of  $\frac{\dd}{\dd t} (g-h,g-h)$ and $\frac{\dd}{\dd t} \left( \theta[g]^{-1} -\theta[h]^{-1} \right)$ together with the  Gronwall  inequality show $g=h$.
\end{proof}

\subsection{Moment equations}\label{sub:moment}

The next step is to study the evolution equations of the moments of $g$
and $G$.

A careful analysis of the moments of $g$ and $G$ will deliver
\begin{enumerate}
\item The rescaling operator $\eta_t$ by requiring that
\begin{equation} \label{covG}
\frac{1}{2}\int G(\bsp) \, \bsp \otimes \bsp \,\dd \bsp = \Id
\end{equation}
holds for all $t\geq 0$.
\item Tightness of $\bsp^2 \,G(\bsp)$.
\end{enumerate}

An easy calculation shows that \eqref{covG} holds if
\begin{equation} \label{teta} \eta_t = T^{-\frac{1}{2}},
\end{equation}
where
$$T=\frac{1}{2}\int g_t(\bsw)\,\bsw \otimes \bsw\,\dd \bsw$$% = (\eta_t^* \eta_t)^{-1}$$
is the Cauchy stress tensor for $g$.
Indeed,
\begin{eqnarray*}
\frac{1}{2}\int  G(\bsp)\,\bsp \otimes \bsp \, \dd \bsp
= \frac{\det \eta_t}{2} \int (\eta_t\, \bsw) \otimes (\eta_t\, \bsw)\,G(\eta_t\, \bsw)\, \dd \bsw =\eta_t\, T\,\eta_t^*
\end{eqnarray*}
as required.

Finally we characterize the long-time behaviour of $G_t$ if $\eta_t=T^{-\frac{1}{2}}$. The result are summarised in the next proposition.

\begin{proposition}\label{prop:star}
Let $G$ be a solution of \eqref{eq:shape} with initial data as in Theorem \ref{FP-thm}    and $\eta_t=T^{-\frac{1}{2}}$,    then the following asymptotics hold for $ t \to \infty$:
\begin{eqnarray} \label{sixstar}
\eta_t  &=& \frac{1}{\mu\, t^{\frac{3}{2}}+O(t)}(\sqrt{3} \bsalpha\otimes \bsalpha + 3(\bsalpha \otimes \bsbeta + \bsbeta \otimes \bsalpha) + 2\mu t \bsbeta \otimes \bsbeta),\\
%&& \frac{1}{\mu\,t^\frac{3}{2}+O(t)} \begin{pmatrix}
%\sqrt{3} & 3\\ 3 & 2 \mu t
%\end{pmatrix},\\
%\label{etadotass}
%\dot \eta_t &=& -\frac{3}{2} \mu^{-1}t^{-\frac52} \begin{pmatrix} \sqrt{3} & 3 \\3 & \frac23 \mu t
%\end{pmatrix},\\
\label{eq:Tinverse}
    T^{-1} &=&\frac{2}{t + O(1)}
\(\frac{6}{(\mu t)^2} \bsalpha\otimes \bsalpha + \frac{3}{\mu t}(\bsalpha \otimes \bsbeta + \bsbeta \otimes \bsalpha) + 2 \bsbeta \otimes \bsbeta\),\\
%&&    \left(\begin{array}{rr} \frac{6}{(\mu t)^2} & \frac{3}{\mu t},\\
% \frac{3}{\mu t} & 2 \end{array}\right),\\
\label{eq:thetainverse}
\theta^{-1} &=& O(t^{-3}),\\ \label{eq:coeff}
\A &=& \left(\dot \eta_t -\mu\eta_t
\bsalpha\otimes \bsbeta\right)\eta_t^{-1}\\ \nonumber  &=&-\frac{1}{2t+O(1)}\( O(1/t^2) \bsalpha\otimes\bsalpha + \sqrt{3}(\bsalpha\otimes \bsbeta-\bsbeta\otimes \bsalpha) + 4\bsbeta \otimes \bsbeta\).
\end{eqnarray}
%\todo[inline]{$O(1/t^2)$ war gemeint .}
Furthermore there exist $c, \bar \lambda, \lambda'>0$ such that for all permissible $G_0\in L^1_6(\R^\di)$ we have that
\begin{equation} \label{alges6}
\int_{\R^\di} \left| (G_t-G^M)(\bsp) \right| \,|\bsp|^6 \, \dd \bsp=O(1+t^{\lambda'})
\end{equation}
and for an open dense set of initial data $G_0\in L^1_6(\R^\di)$ there exists $c>0$ such that for sufficiently large $t$
\begin{equation} \label{alges}
\left|\int_{\R^\di} (G_t-G^M)(\bsp) \,|\bsp|^4 \, \dd \bsp \right|\, > c t^{-\bar \lambda}.
\end{equation}
\end{proposition}

\begin{proof}%[Proof of Proposition \ref{prop:star}]
%%We will demonstrate that the coefficients in \eqref{resGeq} converge as $t \to \infty$, more precisely
%%\begin{eqnarray*}
%%t\,\theta &=& O(t^{-2}),\\
%%t\,\A &=& \frac{\sqrt{3}}{2} (\bsalpha \otimes \bsbeta - \bsbeta \otimes \bsalpha) + 3 \bsbeta \otimes \bsbeta + O(t^{-1}),\\
%%t\,\eta_t^2 &=&  4 \bsbeta \otimes \bsbeta.
%%\end{eqnarray*}
We will first establish formulae~\eqref{sixstar}, \eqref{eq:Tinverse}, \eqref{eq:thetainverse} and \eqref{eq:coeff}
by carefully analyzing the second moments. Formulae~\eqref{alges6} and \eqref{alges} follow from cruder estimates of higher moments.
\subsubsection*{Second moments}
The stress tensor
$T=\frac{1}{2}\int g_t(\bsw)\,\bsw \otimes \bsw\,\dd \bsw$ satisfies a simple ordinary differential equation.
Multiplying \eqref{sred} with $\frac{1}2\bsw \otimes \bsw$ and integrating by parts yields
\begin{eqnarray} \label{Meq}
\frac{\dd T}{\dd t} = \Id - \mu\,\(\bsalpha \otimes \bsbeta\, T+ T\,\bsbeta\otimes \bsalpha\)- \frac{2}{\tr T}T.
\end{eqnarray}
To characterize the asymptotic behavior of $T$ as $t \to \infty$ we define the
rescaled moments $a,b,c$ by the requirement
\begin{eqnarray*}
T = t^{3} \,a\,\bsalpha\otimes \bsalpha + t^{2}\,b\,(\bsbeta \otimes \bsalpha+\bsalpha\otimes \bsbeta)+t\,c\, \bsbeta \otimes \bsbeta.
\end{eqnarray*}
Then \eqref{Meq} reads
\begin{eqnarray*}
&&3\,t^{2} \,a\,\bsalpha\otimes \bsalpha + 2\,t\,b\,(\bsbeta \otimes \bsalpha+\bsalpha\otimes \bsbeta)+c\, \bsbeta \otimes \bsbeta +  t^{3} \,\frac{\dd a}{\dd t}\,\bsalpha\otimes \bsalpha + t^{2}\,\frac{\dd b}{\dd t}\,(\bsbeta \otimes \bsalpha+\bsalpha\otimes \bsbeta)\\&&+t\,\frac{\dd c}{\dd t}\, \bsbeta \otimes \bsbeta\\
&=&  \bsalpha\otimes \bsalpha + \bsbeta\otimes \bsbeta- \mu\,\(2t^2 b\,\bsalpha\otimes \bsalpha + tc\,(\bsalpha\otimes \bsbeta+\bsbeta\otimes \bsalpha)\)\\
&&- \frac{2}{a\,t^3+c\,t}\(t^{3} \,a\,\bsalpha\otimes \bsalpha + t^{2}\,b\,(\bsbeta \otimes \bsalpha+\bsalpha\otimes \bsbeta)+t\,c\, \bsbeta \otimes \bsbeta\)
\end{eqnarray*}
The equations for the individual components read
\begin{eqnarray*}
\bsalpha \otimes \bsalpha: && 0 = 3t^2a+t^3 \frac{\dd a}{\dd t} -1+2t^2\mu b+\frac{2aet^3}{at^3+ct},\\
\bsalpha \otimes \bsbeta: && 0 = 2tb + t^2\frac{\dd b}{\dd t}+\mu t c+\frac{2t^2 b}{a t^3+ct},\\
\bsbeta \otimes \bsbeta: && 0 = c + t\frac{\dd c}{\dd t}-1+\frac{2ct}{at^3+ct}.
\end{eqnarray*}
After rescaling time as well so that $t = \exp(s)$ and $\frac{\dd}{\dd t} = \exp(-s) \frac{\dd}{\dd s}$ equation \eqref{Meq} takes the form
%\begin{eqnarray*}
%3 t^2\,a + t^3\frac{\dd a}{\dd t} &=&-2\mu t^2 b +1-\frac{2 t^3\,a}{t^3 \,a + t\,c},\\
%2t\,b +t^2\frac{\dd b}{\dd t} &=&-\mu t\,c- \frac{2b\,t^2}{t^3 \,a + t\,c}\\
%c+t \frac{\dd c}{\dd t} &=&1-\frac{2 t\,c}{t^3 \,a + t\,c}
%\end{eqnarray*}
%If $a,c\geq 0$ and $b$ are independent of time we find that
%\begin{eqnarray*}
%3a&=& -2 \mu b +t^{-2}- \frac{2a}{t^{2}a+c}\\
%2b&=& -\mu\,c- \frac{2b}{t^2 a+c},\\
%c&=&1-\frac{2 c}{t^2 \,a + c}.
%\end{eqnarray*}
%Sending $t$ to $\infty$ on can easily see that unique solution of the system is given by
%$\bar a= \frac{1}{3} \mu^2$, $\bar b= -\frac{1}{2}\mu$, $\bar c = 1$.
%
%We can get explicit convergence rates by going to exponential time scales: $t = e^s$, then
%$$ \frac{\dd}{\dd t} = e^{-s} \frac{\dd}{\dd s},$$ and we obtain the following dynamic equation for $a,b,d$:
\begin{eqnarray*}
\left(\frac{\dd}{\dd s}-M\right) \left(\begin{array}{c}a\\b\\c\end{array}\right) =
\left(\begin{array}{c}\exp(-2s)\\0\\1\end{array}\right)+ \frac{2}{\exp(2s)\,a +c}\left(\begin{array}{c}a\\b\\c
\end{array}\right)
\end{eqnarray*}
with
$$ M=\left(\begin{array}{rrr}-3 &-2\mu& 0\\
0 &-2 &-\mu\\
0 & 0 &-1  \end{array}\right).$$
%\begin{eqnarray*}
%\frac{\dd a}{\dd s} &=& -3a - 2\mu b + e^{-2s}-\frac{2a}{e^{2s}a+c},\\
%\frac{\dd b}{\dd s} &=& -2b-\mu c -\frac{2b}{e^{2s}a+c},\\
%\frac{\dd c}{\dd s} &=& 1-c-\frac{2c}{e^{2s}a+c}.
%\end{eqnarray*}

As the spectrum of $M$ is given by $\lambda_1 = -1$, $\lambda_2 = -2$, $\lambda_3 = -3$ with corresponding eigenvectors $\bsv_1 = (1, 0, 0), \bsv_2 = (-2\mu, 1, 0), \bsv_3 = (\mu^2,-\mu,1)$
a simple application of the variation of constants formula delivers the asymptotic result
$$ \left(\begin{array}{c} a\\b\\c\end{array}\right) = M^{-1}\left(\begin{array}{c} 0\\0\\1\end{array}\right)+ O(\exp(-s)) = \left(\begin{array}{c} \frac{1}{3} \mu^2\\
-\frac{1}{2}\mu\\ 1\end{array}\right) +O(t^{-1}).$$
This implies that the stress tensor admits the asymptotic result
\begin{equation}
\label{Tassymp}
T = (t+O(1))\,T_\infty, \quad t \to \infty
\end{equation}
where
$$ T_\infty =
\frac{1}{3}(t\mu)^2 \bsalpha\otimes \bsalpha
-\frac{1}{2}t \mu\, (\bsalpha\otimes \bsbeta +\bsbeta\otimes \bsalpha)
+ \bsbeta \otimes \bsbeta.$$
We can bootstrap this step by plugging \eqref{Tassymp} into \eqref{Meq}. This shows that
the function
$t \to T - t\, T_\infty$ is differentiable and satisfies
\begin{equation}
\frac{\dd}{\dd t}(T-t\, T_\infty) = O(t^{-1}), \quad t \to \infty.
\end{equation}
These asymptotics give results for $\eta_t$, $\eta_t^{-1}$, $\dot \eta_t$ and $T^{-1}$.
Then as $\theta^{-1} = (\tr T)^{-1}$, this implies \eqref{eq:Tinverse} and one also has the asymptotic result for $\A$.

\subsubsection*{Fourth and sixth moments}
We are now assuming that the related initial data for $G_0$ satisfy $g_0\in L^1_6(\R^\di)$ and using \eqref{eq:fokkmild} we can see that higher moments up to order $6$ are well-defined for finite times, these moments satisfy similar ODEs. These will later allow us to choose suitable initial data for lower estimates on the rate of decay.

Letting
\begin{equation}\label{eqn:hij}
  h_{ij}(t)= \int_{\R^\di} (G_t(\bsp)- G^M(\bsp))p_1^i p_2^j \, \dd \bsp,
\end{equation}
we obtain ordinary differential equations, which only depend on modes of the same or lower order. The moment of order $0$ and $1$ (mass and momentum) are preserved by  Proposition~\ref{fpcons}
for the evolution of $g$, in the same way this also follows for $G$, where the momentum is assumed to be $0$, i.e
\begin{equation}\label{eq:h0}
  h_{00}(t)\equiv 0=h_{10}(t)\equiv h_{01}(t)\equiv 0
\end{equation}
The rescaling $\eta_t$ is defined such that \eqref{covG} holds, i.e.
\begin{equation}\label{eq:h2} h_{20}(t)\equiv h_{02}(t) \equiv h_{11}(t)\equiv 0.
\end{equation}
For the higher moments $h_{ij}$ with $i+j>2$ we obtain using integration by parts
\begin{align}
\label{eq:hode}
  \frac{\dd}{\dd t}{h_{ij}}(t)&= -\bigl( \frac{i+j}\theta -i  \A_{11}-j\A_{22} \bigr)\, h_{ij} +j\A_{12}\, h_{i+1 \, j-1}+i\A_{21} \, h_{i-1 \, j+1}  \\
&\nonumber +  i (i-1) (\eta_t^2)_{11}\, h_{i-2 \,j}  + ij (  (\eta_t^2)_{12} + (\eta_t^2)_{21}) \,h_{i-1 \, j-1} +j (j-1) (\eta_t^2)_{22} \, h_{i \,j-2},
\end{align}
where we assumed without loss of generality that $\bsalpha=(1,0)$ and $\bsbeta = (0,1)$.
Using the information on the coefficients in Proposition~\ref{prop:star} we write
the moment equations in matrix notation:
\begin{align}
\label{eq:hode1}
  \frac{\dd}{\dd t}{h_{ij}}(t)&= t^{-1}\sum_{k,l}(N_{ijkl} +O(t^{-1}))\,h_{kl(t)} \quad t\gg 1,
\end{align}
where the operator $N$ is defined by
\begin{align*}
N_{ij kl} &= \left\{\begin{array}{rl} -2j & \text{ if }(i,j) =(k,l),\\
-\frac{\sqrt{3}}{2}j & \text{ if }  (i,j) =(k+1,l-1),\\
\frac{\sqrt{3}}{2} i& \text{ if } (i,j) =(k-1,l+1),\\
2j(j+1)& \text{ if } (i,j) = (k,l+2),\\
0 & \text{ else}.
\end{array}\right.
\end{align*}
with the convention that $h_{ij} =0$ if $i+   j    \leq 2$. The moments of odd order can all be chosen to be $0$, which is preserved by \eqref{eq:hode}.
%\begin{equation}\label{eq:4} \frac{d}{dt}\left( \begin {array}{c} h_{40}\\h_{31} \\h_{22}\\h_{13}\\h_{04}\end{array} \right)=  \frac{1}{t+O(1)} \left( \begin {array}{ccccc} \frac{1}{O(t^2)}&2\,\sqrt {3}&0&0&0
%\\ \noalign{\medskip}-\sqrt {3}/2&-2&3\,\sqrt {3}/2&0&0
%\\ \noalign{\medskip}0&-\sqrt {3}&-4&\sqrt {3}&0\\ \noalign{\medskip}0
%&0&-3\,\sqrt {3}/2&-6&\sqrt {3}/2\\ \noalign{\medskip}0&0&0&-2\,
%\sqrt {3}&-8\end {array} \right)\left( \begin {array}{c} h_{40}\\h_{31} \\h_{22}\\h_{13}\\h_{04}\end{array} \right) +O(t^{-2}), \quad t\ll 1.
%+ \left( \begin {array}{c} \frac{288}{(t+O(1))(\mu t)^2} \\ \frac{72}{(t+O(1))(\mu t)} \\  \frac{16}{t+O(1)} \\ \frac{72}{(t+O(1))(\mu t)^2} \\ \frac{96}{t+O(1)}  \end{array} \right)
%\end{equation}
Now rescaling time $t=\exp(s)$ and $\frac{\dd}{\dd t}=\exp(-s) \frac{\dd}{\dd s}$ equation \eqref{eq:hode1} becomes
\begin{equation}\label{eq:4res}
  \left( \frac{\dd}{\dd s} -N+ O(e^{-s})\right) h  =0  \quad s\gg 1,
  %\left( \begin {array}{c} O(\exp(-s))  \\ O(\exp(-s)) \\  16+O(\exp(-s))  \\ O(\exp(-s)) \\ 96+O(\exp(-s))  \end{array} \right),
\end{equation}
where the operator $N$ if a lower triangular form. Consider now the truncated operator
\[ (N_{ijkl})_{i+j=k+l=4} =\left( \begin {array}{ccccc} 0&2\,\sqrt {3}&0&0&0
\\ \noalign{\medskip}-\sqrt {3}/2&-2&3\,\sqrt {3}/2&0&0
\\ \noalign{\medskip}0&-\sqrt {3}&-4&\sqrt {3}&0\\ \noalign{\medskip}0
&0&-3\,\sqrt {3}/2&-6&\sqrt{3}/2\\
\noalign{\medskip}0&0&0&-2\,
\sqrt {3}&-8\end {array} \right). \]
It has only eigenvalues  with negative real parts by the Routh-Hurwitz stability criterion. Using the variation of constants formula we obtain
 \begin{equation}\label{eq:4ass}
  (h_{ij})_{i+j=4}
%  \left( \begin {array}{c} h_{40}\\h_{31} \\h_{22}\\h_{13}\\h_{04}\end{array} \right) = %-N^{-1}\left( \begin {array}{c}0 \\ 0 \\  16 \\ 0 \\ 96  \end{array} \right)
   %\left(\begin{array}{c} 12\\ 0 \\ 4 \\ 0 \\ 12 \end{array} \right)+
 =O(\exp(-\bar \lambda s))
 \end{equation}
for some    $\tilde \lambda>0$.  Letting $u(s)=\exp(-Ns) h(s)$ changes \eqref{eq:4res} into
\[ \left( \frac{\dd}{\dd s} + O(e^{-s})\right) u  =0.  \]
As the $O(e^{-s})$ term is integrable, bounded initial data $u(0)$ will remain bounded in norm from above and below for all times $s>0$. Hence
 there is some $\bar \lambda \geq \tilde{\lambda}$
\begin{equation}\label{eq:h4lower}
  \liminf_{s\to \infty}\exp(\bar \lambda s)\,
  \left|%\left( \begin {array}{l} h_{40}\\h_{31} \\h_{22}\\h_{13}\\h_{04}\end{array} \right)
  (h_{ij})_{i+j=4}\right|
  > 0
\end{equation}
   for all nonzero initial data in \eqref{eq:hode}. All initial data $G_0 \in L^1_6(\R^2)$ with a different tensor of fourth order moments --after the coordinate change $\eta_0$-- compared to $G^M$  will have then have nonzero initial data in \eqref{eq:hode}, this set is open and dense in the set of possible initial data in $L^1_6(\R^2)$.    Transferring this back to time
$t$ gives the algebraic estimate \eqref{alges} for some $c>0$ which depends on the initial value and all $t$ large enough.

A less detailed calculation for the vector $h$ of sixth moments gives then
\[ \frac{\dd}{\dd t} h = N^{(6)}(t) h + O(\frac{1}{1+ t^\alpha}), \]
where $\alpha>1$ and $\|N^{(6)}(t)\| = O(\frac{1}{1+ t})$ as $t \to \infty$. After transforming to time $s$ as above we obtain a constant matrix plus some exponentially small error terms, this is enough using Gronwall's inequality to conclude that  $h$  grows at most exponentially in $s$ and after transforming  back to time $t$  that $h$ grows  at most algebraically with rate $t^{\lambda'}$ such that --without loss of generality-- $\lambda'\geq 0$, this then yields \eqref{alges6}.
\end{proof}

% are summarised in the next proposition.

\subsection{Asymptotics of shape equation and hypocoercivity}\label{sec:shape}

With the asymptotic information on the coefficients we can study the shape equation~\eqref{eq:shape}.

\begin{lemma} \label{Mwellfp}
The Maxwellian $G^M(\bsp)= \frac{1}{4\pi}\exp(-\frac{1}{4} |\bsp|^2)$ is a stationary solution to~\eqref{eq:shape}.
\end{lemma}
\begin{proof}

Substituting $G^M$ into \eqref{eq:shape} and observing that $\nabla G^M(\bsp) = -\frac{1}{2}G^M(\bsp)\,\bsp$	
one finds that $G^M$ is stationary if and only if
\begin{eqnarray} \label{GMstatc}
\tr\(\theta^{-1}\Id - \A -\frac{1}{2}\eta_t^2\)\,G^M(\bsp)
-\frac{1}{2}\bsp\cdot \(\theta^{-1}\Id-\frac{1}{2}(\A+\A^*)-\frac{1}{2}\eta_t^2\)\bsp\,G^M(\bsp) = 0.
\end{eqnarray}
Next, recall that by \eqref{covG} the covariance matrix of $G$ is constant, i.e. $\frac{\dd}{\dd t} \int \bsp \otimes \bsp \, G(\bsp) \, \dd \bsp = 0$. Hence, after
multiplication of \eqref{eq:shape} with $\frac{1}{2}\bsp\otimes \bsp$ and integration by parts one obtains that
\begin{eqnarray}
% \nonumber
\A+\A^* + \eta_t^2 &=&2\,\theta^{-1} \Id.
\label{resmeq2}
\end{eqnarray}
Clearly both terms on the left-hand side \eqref{GMstatc} vanish thanks to \eqref{resmeq2}.
\end{proof}
%As $\eta_t$ is characterized by $\eta_t^2\,T_\Id = \Id$ one obtains the following differential equation for $\eta_t$:
%\begin{eqnarray} \nonumber
%0&=&\dot \eta_t \,T_\Id\,\eta_t + \eta_t\,T_\Id\, \dot \eta_t + \eta_t\,\(\Id - \frac{2}{\tr T}T -\mu\,\( %\bsalpha \otimes \bsbeta\, T + T\,\bsbeta\otimes\bsalpha\)\)\,\eta_t .\\
%\Leftrightarrow 0&=&\dot \eta_t \, \eta_t^{-1} + \eta_t^{-1}\, \dot \eta_t+ \eta_t^2 -\frac2{\theta}\,\Id- \mu \(\eta_t\, e_1 \otimes e_2\, \eta_t^{-1} + \eta_t^{-1} \,e_2\otimes e_1\,\eta_t\).\label{resmeq3}
%\end{eqnarray}
%The last equation coincides with \eqref{resmeq2}, as it should.
%

Although Lemma~\ref{Mwellfp} provides a candidate for the attractor of the evolution it is not
obvious that $\lim_{t\to \infty} G_t = G^M$ holds (in any norm). The reason is that by \eqref{sixstar} the coercivity constant of the dissipation operator $\nabla \cdot\eta_t^2 \nabla$ diverges as $t\to \infty$. To show the convergence we use a nonautonomous perturbation to the theory of hypocoercivity \cite{VillHyp} combined with a priori estimates for the full equations.

It is advantageous to rescale time by defining
$$ \tilde G_s = G_{\exp(s)}.$$
As $t = \exp s$ and $\frac{\dd t}{\dd s} = t$ the density $\tilde G$ satisfies the rescaled equation
\begin{equation} \label{resGeq}
t \,\partial_t G = \partial_s \tilde G = t \,\nabla\cdot\(\tilde G(\bsp)\,(\theta^{-1} \Id - \A)\bsp + \eta_t^2\nabla \tilde G\).
\end{equation}
The coefficients in \eqref{resGeq} are controlled by proposition~\ref{prop:star}.
Next we rewrite \eqref{resGeq} as a density with respect to $G^M$, then we obtain for $G_t = u_t\,G^M$,
\begin{equation}\label{eq:u}
  \partial_s u =\nabla_\bsp \cdot (T^{-1}\nabla u)+   \nabla u \cdot \(\theta^{-1}\Id - \A -T^{-1}\)\bsp
\end{equation}
We split the last equation into an autonomous main part using \eqref{eq:Tinverse},\eqref{eq:thetainverse} and \eqref{eq:coeff} that provide bounds on  some decaying perturbation.
\begin{eqnarray} \label{asspde}
\partial_s u&=&4 \tr (\bsbeta \otimes \bsbeta \nabla^2 u)+\nabla u\cdot \(\frac{\sqrt{3}}{2}\(\bsalpha \otimes \bsbeta- \bsbeta \otimes \bsalpha\) -2  \bsbeta \otimes \bsbeta\)\bsp
\\ \nonumber &&+ \exp(-s) \Bigl\{
\nabla u\cdot \(C_1 \bsalpha \otimes \bsbeta+C_2 \bsbeta \otimes \bsalpha + C_3 \bsbeta \otimes \bsbeta +C_4    \bsalpha \otimes \bsalpha \)\bsp\\ \nonumber &&
+\(C_5\(\bsalpha \otimes \bsbeta+\bsbeta \otimes \bsalpha\) + C_6 \bsbeta \otimes \bsbeta +C_7  \exp(-s)  \bsalpha \otimes \bsalpha \)\nabla^2 u)
 \Bigr\}
\end{eqnarray}
for some appropriate uniformly bounded non-autonomous coefficients $C_i$ for $ i=0,\ldots,7$.

For the long-term convergence of solution of the shape equations we use  Villani's concept of Hypocoercivity \cite{VillHyp}. Consider a separable Hilbert space $\mathcal H$ with inner product $\langle\cdot, \cdot \rangle$, which will be $L^2(\R^\di, \dd G^M)$ in our case. Let $A=(A_1, \ldots,A_m)$ be an unbounded operator for some $m \in \N$ with domain $D(A)$ and let $B$ be an unbounded antisymmetric operator with domain $D(B)$. The theory reduces the convergence to equilibrium of the nonsymmetric operator $L=A^*A+B$, which is not coercive in our case, to the study of the symmetric operator $A^*A+C^*C$ using the commutator $C=[A,B]$. Under appropriate conditions this operator is coercive, which then implies convergence in the abstract Sobolev space $\mathcal H^1$ with norm
$\|h\|_{\mathcal{H}^1}^2 =\langle h,h \rangle +  \langle Ah,Ah \rangle+ \langle Ch,Ch \rangle$, in our case this will coincide with $H^1(\R^\di,\dd G^M)$. The simplest form of the theory is enough for our example and it is stated next.
\begin{theorem}\cite[Theorem 18]{VillHyp}\label{thm:hypoco} With the above notation, consider a linear operator $L=A^*A+B$ with $B$ antisymmetric, and define the commutator $C:=[A,B]$. Assume the existence of constants $\alpha, \beta$ such that
\begin{enumerate}
  \item $A$ and $A^*$ commute with $C$; $A_i$ commutes with each $A_j$;
  \item $[A,A^*]$ is $\alpha$-bounded relatively to $I$ and $A$;
  \item $[B,C]$ is $\beta$-bounded relatively to $A,A^2,C$ and $AC$
\end{enumerate}
Then there is a scalar product $\langle\langle\cdot, \cdot \rangle\rangle$ on $H^1(\R^\di,\dd G^M)/\mathcal{K}$, which defines a norm equivalent to the $H^1$ norm, such that
\begin{equation}\label{eq:hyp}
  \forall h \in H^1/\mathcal{K}, \quad \langle\langle h,Lh \rangle\rangle \geq K(\|Ah\|^2 +\|Ch\|^2)
\end{equation}
for some constant $K>0$ depending on $\alpha$ and $\beta$. If in addition
\[ A^* A +C^* C \mbox{ is $\kappa$ -coercive}\]
for some $\kappa>0$, then there exists a constant $\lambda>0$, such that
\[ \forall h \in H^1/\mathcal{K}, \quad \langle\langle h,Lh \rangle\rangle \geq \lambda \langle\langle h,h \rangle\rangle.
 \]
In particular, $L$ is hypocoercive in $H^1/\mathcal{K}$, there is a $c <\infty$
\[\| \exp(-tL)\|_{H^1/\mathcal{K} \to H^1/\mathcal{K}} \leq c \exp(-\lambda t), \]
where both $\lambda $ and $c$ only depend on upper bounds for $\alpha$ and $\beta$ and lower bounds on $\kappa$.
\end{theorem}

The last theorem is used to show that the leading order  of \eqref{asspde}, i.e. its autonomous part, is hypocoercive. Then we use the similar splitting
$L_s=A^*_sA_s+B_s$ for the full equation to obtain a priori estimates. Both ingredients will then combined via a Duhamel formula to provide the convergence
result for \eqref{asspde}.

\begin{proposition}\label{prop:contract}
  The autonomous part of \eqref{asspde} given by
  \begin{equation}\label{autpde}
  \partial_s u=4 \tr (\bsbeta \otimes \bsbeta \nabla^2 u)+\nabla u\cdot \(\frac{\sqrt{3}}{2}\(\bsalpha \otimes \bsbeta- \bsbeta \otimes \bsalpha\) -2  \bsbeta \otimes \bsbeta\)\bsp
  \end{equation}
  defines a contraction in time in $H^1(\R^\di,\dd G^M)/\mathcal{K}$, where $\mathcal{K}=\mbox{span}\{ 1\}$, i.e. there exist $c,\lambda >0$ such that
 \begin{equation}\label{eq:Lcontract}
  \| \exp(-sL)\|_{H^1/\mathcal{K} \rightarrow H^1/\mathcal{K}} \leq c \exp(-\lambda s).
\end{equation}
\end{proposition}
\begin{proof}
 Consider $L^2(\R^\di,\dd G^M)$ with inner product $\langle u,v \rangle = \int_{\R^\di} u(\bsp)v(\bsp) \,G^M(\bsp) \dd \bsp$. We are now writing \eqref{autpde} in Villani's notation
\begin{equation}\label{eq:ut-Lu}
  \partial_s u +L u= 0 \mbox{ with } L=A^*A +B \mbox{ and } B \mbox{ antisymmetric in }L^2(\R^\di,\dd G^M).
\end{equation}
Choosing coordinates and identifying the canonical basis vectors $\bse_1$  $\bse_2$  with $\bsalpha$ and $\bsbeta$ respectively, we let
\begin{equation}\label{eq:AB}
  (A u)(p_1,p_2) = 2 \partial_2 u(p_1,p_2)  \qquad (B u)(p_1,p_2) =-\frac{\sqrt{3}}{2}(p_2 \partial_1 u(p_1,p_2)  -p_1 \partial_2 u(p_1,p_2))
\end{equation}
Then we obtain the adjoint  $A^*$  of $A$ in  $L^2(\R^\di,\dd G^M)$ by integration by parts in the inner product.
\begin{equation}\label{eq:AstarB}
  (A^* u)(p_1,p_2) = -2 \partial_2 u(p_1,p_2) + p_2 u(p_1,p_2),
\end{equation}
while $B$ is antisymmetric, such that \eqref{eq:ut-Lu} is a reformulation of \eqref{autpde}. We now check the assumptions of theorem \ref{thm:hypoco}. We observe
$C:=[A,B]= -\sqrt{3}\partial_1$ and then (i) $A$ and $A^*$ commute with $C$. Furthermore (ii) holds as $[A, A*]=2 I$. The commutator $[B,C]= -\frac{3}{2}\partial_2$ is relatively bounded by $A$, hence  (iii) holds.  The general results imply then $\mathcal{K}= \mbox{Ker} L= \mbox{Ker} A \cap \mbox{Ker} B$ consists of constants only.
In addition $A*A +C*C= -4 \partial_2^2 +2p_2 \partial_2 -3 \partial_1^2$ is coercive on $L^2(\R^\di,\dd G^M)$ using  a Poincar\'e inequality as in \cite[Thm A.1]{VillHyp}. Then Theorem \ref{thm:hypoco} implies there exist positive constants  $\lambda$ and $c$ such that \eqref{eq:Lcontract} holds, completing the proof.
\end{proof}

To obtain a priori estimates, equation \eqref{asspde} is rewritten in the form of the last proposition with time-dependent operators $A_s$ and $B_s$.
\begin{equation}\label{asspde2}
  \partial_s u =-L_s u= - A^*_s A_s u-B_s u
\end{equation}
where
\begin{align}\label{As}
  A_s &= \eta_s \nabla  \\
  A^*_s. &=- \nabla \cdot (\eta_s . )+ \frac 1 2 \bsp \eta_s . \label{Astars}\\
  B_su &=-\nabla u \cdot \left( \theta^{-1}\Id - \A -\frac{1}{2}T^{-1} \right) \bsp \label{Bs}.
\end{align}
Then by \eqref{GMstatc}
\[ B^*_su= \nabla u \cdot \left( \theta^{-1}\Id - \A -\frac{1}{2}T^{-1} \right) \bsp + \frac 1 2 \bsp \left( \theta^{-1}\Id - \A -\frac{1}{2}T^{-1} \right) \bsp =-B_su,\]
such that $B_s$ is anti-symmetric.
\begin{lemma}\label{lem:apriori}
Let $u_s$ be the solution \eqref{asspde} obtained from rescaling the solution in Proposition~\ref{prop:solFP}, then there is
$K_*>0$ such the a priori estimates holds for  all $s\geq 1$.
\begin{equation}\label{eq:apriori}
  \| u_s\|_{H^1(\R^\di,G^M)} +  \| \nabla u_s \|_{H^1(\R^\di,G^M)} + \| \nabla^2 u_s \|_{H^1(\R^\di,G^M)}  \leq K_*.
\end{equation}
\end{lemma}
\begin{proof}
Using the form in \eqref{asspde2} we estimate  the time derivative of  the $L^2$ norm
\begin{align*} \partial_s\langle u_s,u_s\rangle =&-2\langle L_s u_s,u_s\rangle =2\langle-A^*_s A_s u_s+B_s u_s,u_s\rangle\\
=&2\langle A_s u_s,A_s u_s\rangle\leq 0. \end{align*}
The derivatives of $\partial_1 u$ and $\partial_2 u$ with respect to  $p_1$ and $p_2$ satisfy equations similar to \eqref{asspde2}.
\begin{align}\label{u1}
  \partial_s \partial_1 u_s &= - A^*_s A_s \partial_1 u_s-B_s \partial_1 u_s + \nabla u_s \cdot \(\theta^{-1}\Id - \A -T^{-1}\)\bse_1\\
  \partial_s \partial_2 u_s &= - A^*_s A_s \partial_2 u_s-B_s \partial_2 u_s + \nabla u_s \cdot \(\theta^{-1}\Id - \A -T^{-1}\)\bse_2 \label{u2}
\end{align}
This yields with the anti-symmetry of $B_s$
\begin{align*}
   \partial_s\langle \nabla u_s, \nabla u_s\rangle &= \partial_s \left(\langle \partial_1 u_s, \partial_1 u_s\rangle +  \langle \partial_2 u_s, \partial_2
   u_s\rangle\right)\\ &=2\langle A_s\partial_1 u_s, A_s\partial_1 u_s\rangle +2
    \langle \nabla u_s \cdot \(\theta^{-1}\Id - \A -T^{-1}\) \bse_1, \partial_1 u_s\rangle  \\
   & + 2\langle A_s\partial_2 u_s,A_s\partial_2 u_s\rangle  +
    2\langle \nabla u_s \cdot \(\theta^{-1}\Id - \A -T^{-1}\) \bse_2, \partial_2 u_s\rangle  \\
  & \leq 2\langle \nabla u_s \cdot \(\theta^{-1}\Id - \A -T^{-1}\) \bse_1, \partial_1 u_s\rangle + \langle \nabla u_s \cdot \(\theta^{-1}\Id - \A -T^{-1}\) \bse_2, \partial_2 u_s\rangle \\
  &= 2  \langle \nabla u_s \cdot \(\theta^{-1}\Id - \A -T^{-1}\),\nabla u_s \rangle \\
  & \leq C \exp(-s) \langle \nabla u_s, \nabla u_s\rangle
\end{align*}
where autonomous terms in \eqref{asspde} either cancel or have a sign, the form of the non-autonomous first-order terms yields the remainder. Then the Gronwall inequality shows that
\[ \langle \nabla u_s, \nabla u_s\rangle \leq \exp\left(C \int_1^\infty \exp(-\sigma) \dd \sigma\right) \langle \nabla u_1, \nabla u_1\rangle=  \exp\left(C/\mathrm{e}\right) \langle \nabla u_1, \nabla u_1\rangle     \]
remains bounded for all times $s>1$. A similar argument also holds for higher derivatives. We derive differential equations for higher derivatives:
\begin{align}\label{u11}
  \partial_s \partial_1^2 u_s &= - A^*_s A_s \partial^2_1 u_s-B_s \partial^2_1 u_s + 2 \nabla \partial_1 u_s \cdot \(\theta^{-1}\Id - \A -T^{-1}\)\bse_1\\
  \partial_s \partial_2 \partial_1 u_s &= - A^*_s A_s \partial_2\partial_1 u_s-B_s \partial_2\partial_1 u_s + \nabla \partial_1 u_s \cdot \(\theta^{-1}\Id - \A -T^{-1}\)\bse_2 \label{u21}\\
  & + \nabla \partial_2 u_s \cdot \(\theta^{-1}\Id - \A -T^{-1}\)\bse_1 \nonumber \\
  \partial_s \partial^2_2 u_s &= - A^*_s A_s \partial^2_2 u_s-B_s \partial^2_2 u_s + 2 \nabla \partial_2 u_s \cdot \(\theta^{-1}\Id - \A -T^{-1}\)\bse_2 \label{u22}
\end{align}
These equations yield due the properties of  $A_s$ and $B_s$
\begin{align*}
  & \frac 1 2 \partial_s \left(\langle \partial^2_1 u_s, \partial^2_1 u_s\rangle +2 \langle \partial_2 \partial_1 u_s, \partial_2 \partial_1 u_s\rangle+  \langle \partial^2_2 u_s, \partial^2_2 u_s\rangle\right)   \\
  &\leq 2 \langle \nabla \partial_1 u_s \cdot \(\theta^{-1}\Id - \A -T^{-1}\),\nabla \partial_1 u_s \rangle + 2\langle \nabla \partial_2 u_s \cdot \(\theta^{-1}\Id - \A -T^{-1}\),\nabla \partial_2 u_s \rangle\\
  & \leq C \exp(-s) \left( \langle \nabla \partial_1 u_s, \nabla \partial_1 u_s\rangle  +  \langle \nabla \partial_2  u_s, \nabla \partial_2 u_s\rangle \right),
\end{align*}
where the autonomous terms in \eqref{asspde} again either cancel or have a sign, the form of the non-autonomous first-order terms yields the remainder. Using the Gronwall inequality yields
a bound on the second derivatives after an initial regularisation, e.g. for $s>1$. Deriving similar equations for third derivatives and estimating
\[ \partial_s \left(\langle \partial^3_1 u_s, \partial^3_1 u_s\rangle +3 \langle \partial^2_2 \partial_1 u_s, \partial^2_2 \partial_1 u_s\rangle
+3 \langle \partial_2 \partial^2_1 u_s, \partial_2 \partial^2_1 u_s\rangle +  \langle \partial^3_2 u_s, \partial^3_2 u_s\rangle\right) \]
yields the final required estimate.
\end{proof}

It remains to establish the convergence of $G_t$ to $G^M$ in $L^1$. It suffices to show that $u_s \to 0$ in $L^2(\R^\di,\dd G^M)/\mbox{span}(1)$. Indeed, we show convergence of $u$ in the stronger  $H^1(\R^\di,\dd G^M)$  norm. Using the Duhamel principle for the equation
\begin{equation}\label{duh1}
  \partial_s u_s=-Lu_s -(L_s-L)u_s
\end{equation}
with $L$ as in  \eqref{eq:ut-Lu} and $L_s$ as in \eqref{asspde2}. We
starting  from the positive time $1$ for $s>1$ to guarantee uniform bounds for higher derivatives as in Lemma \ref{lem:apriori}.
\begin{align*}
  u_s & =\exp(-L_{s-1})\,u_1- \int_1^s \exp(-L_{s-\sigma})(L_\sigma-L)\,u_\sigma\, \dd \sigma
\end{align*}
Then using the error estimates of $L_\sigma-L$ in \eqref{asspde}, Lemma \ref{lem:apriori} together with the contraction
property of $L$ in $H^1(\R^\di,\dd G^M)$ as in proposition \ref{prop:contract}  yields for $s>1$
 \begin{align*}
 &\|u_s\|_{H^1(\R^\di,\dd G^M)}\\& \leq \exp(-\lambda (s-1))\, \|u_1\|_{H^1(\R^\di,\dd G^M)}
 + \int_1^t \|\exp(-L_{s-\sigma})\|_{H^1 \to H^1}\|(L_\sigma-L)u_\sigma\|_{H^1(\R^\di,\dd G^M)}\,  \dd \sigma\\
 &\leq \exp(-\lambda (s-1)) \,\|u_1\|_{H^1(\R^\di,\dd G^M)} + \int_1^t \exp(-\lambda (s-\sigma))\, C\, \exp(-\sigma)\, K^* \,\dd \sigma\\
 &\leq C s \exp(-\min\{\lambda,1\}\, s)
 \end{align*}
for some bounded $C$ only depending on the $L^1 \cap L^\infty$ norms of the initial data due to initial regularisation.
Undoing the change of time from $t$ to $s$ we also see the rate of convergence is bounded by any  algebraic order greater than $\min\{1, \lambda\}$.

\subsection{Lower estimates using higher moments}\label{subsec:lower}

The estimates on the higher moments in the last proposition yield the lower estimates \eqref{eq:ratelower} in the following way for almost all initial data.

Let $B_R$ the ball of radius $R$ in $\R^\di$,  we first note that for all $R>0$ using \eqref{alges6}
\begin{align*}\nonumber
  &\left|\int_{\R^\di} (G(\bsp)-G^M(\bsp))\,|\bsp|^4 \, \dd \bsp \right|\\
  \leq & R^4 \int_{B_R} |G(\bsp)-G^M(\bsp)| \, \dd \bsp +R^{-2} \int_{\R^\di\setminus B_R} |G(\bsp)-G^M(\bsp)| \,|\bsp|^6 \, \dd \bsp\\
 \leq & R^4 \int_{\R^\di} |G(\bsp)-G^M(\bsp)| \, \dd \bsp + \frac{C}{R^2}t^{\lambda'}.
\end{align*}
Note that $R$ may depend on $t$ in the above estimate.

Then \eqref{alges} implies with the choice $R =R(t)=t^{\bar\lambda+\lambda'/2}$ that
\begin{align} \label{eq:mom46}
  \left|\int_{\R^\di} (G_t(\bsp)-G^M(\bsp))|\bsp|^4 \, \dd \bsp \right| \geq 2 \frac{C}{R^2}t^{\lambda'}
\end{align}
for $t$ large enough as the right  hand side is    $ O(t^{-2 \bar \lambda})$.    Then we obtain that
\begin{align*}\liminf_{t\to \infty} t^{5 \bar \lambda+2 \lambda'}\|G_t-G^M\|_{L^1(\R^\di)}&=\liminf_{t\to \infty}  t^{\bar \lambda} R^4\|G_t-G^M\|_{L^1(\R^\di)} \\
%&\geq \liminf_{t\to \infty} t^{\lambda'} \left[ \int_{\R^\di} |G_t(\bsp)-G^M(\bsp)|(1+|\bsp|^4)\, \dd \bsp - {C} \,R^{-2}\, t^{\lambda''}\right]\\
 &\geq \liminf_{t \to \infty}\,t^{\bar \lambda} \left[\left|  \int_{\R^\di} (G_t(\bsp)-G^M(\bsp))\,|\bsp|^4 \, \dd \bsp\right| -C  \,R^{-2}\, t^{\lambda'}\right]\\
  &\geq \liminf_{t \to \infty}t^{\bar \lambda}\frac{1}{2}\left|  \int_{\R^\di} (G_t(\bsp)-G^M(\bsp))\,|\bsp|^4 \, \dd \bsp\right| >0,
\end{align*}
where the penultimate  estimate is due to \eqref{eq:mom46} for sufficiently large $t$.

\subsection{Summary of  the proof of Theorem \ref{FP-thm}}\label{subsec:final}

This completes the proof.
The statements on the regularity of $G$ follow from Proposition~\ref{prop:solFP}. The properties of the rescaling operator $\eta_t$ are given in proposition \ref{prop:star}. The convergence was shown at the end of subsection \ref{sec:shape}. The lower estimate \eqref{eq:ratelower} was given in subsection \ref{subsec:lower}.

\section{The Boltzmann case}
Now we consider the case where $f$ satisfies the Boltzmann equation
\begin{eqnarray}
\label{BE}
\partial_t f +\nabla_\bsz f\cdot\bsw= Q[f]  && \text{ for } \bsz, \bsw \in  \R^\di,
\end{eqnarray}
together with the    $S$-objectivity condition     \eqref{shear}.
For simplicity the collision operator $Q$ is assumed to be the hard-sphere kernel
\begin{eqnarray*}
	Q[f](v) &=& \int_{S^{1}}\int_{\R^\di}(f_*f_*'-f\,f')\;
	(v-v')\cdot \nu_+\,\dd v'\, \dd \nu
\end{eqnarray*}
We repeat the reduction steps in Section~\ref{sec:FP} and obtain the
equivalent of equation \eqref{sred}:
\begin{equation} \label{be}
\left\{ \begin{array}{rl} \partial_tg& = \mu\,\nabla\cdot(g\,\bsalpha\otimes \bsbeta \bsw )+ Q[g],\\[0.5em]
g|_{t=0} &= g_0.
\end{array} \right.
\end{equation}
where $g_t= g_t(\bsw)$ and $Q$ is unchanged (acts on $\bsw$). As before we define the kinetic energy by
$$ \theta[g] = \frac{1}{2} \int_{\R^\di} |\bsw|^2 \,g(\bsw)\,\dd \bsw.$$
The energy $\theta$ is conserved if and only if $\mu=0$. A quantitative version of this
observation delivers the existence and uniqueness of solutions for all time.
For some time-dependent transformation $\eta_t \in \R^{d\times d}$ %_\mathrm{sym}$
we repeat the notation \eqref{renorm} and define
\begin{equation*}
\begin{cases}
\bsp =\eta_t \,\bsw,\\
G(\bsp) = \det \eta_t\, g(\eta_t^{-1} \bsp).
\end{cases}
\end{equation*}
Then the collision operator $Q$ can be written in terms of a rescaled collision operator
$$ Q_{\eta_t}[G] = \det\eta_t\,Q[g],$$
where
\begin{eqnarray*}
Q_{\eta_t}[G] &=& \int_{S^{d-1}}\int_{\R^\di}(G_*G_*'-G\,G')\;
[\nu\cdot\eta_t^{-1}(\bsp-\bsp')]_+\,\dd \bsp'\, \dd \nu,\\
\bsp_* &=& \bsp - \eta_t \nu \otimes \nu\eta_t^{-1}(\bsp-\bsp'),\\
\bsp_*'&=& \bsp' + \eta_t \nu \otimes \nu\eta_t^{-1}(\bsp-\bsp').
\end{eqnarray*}

The function $G$ satisfies the equation
\begin{equation} \label{betranc}
\partial_t G= Q_{\eta_t}[G]-\nabla_{\bsp}\cdot(G \A\bsp),
\end{equation}
with $\A=(\dot \eta_t-\mu \eta_t \bsalpha\otimes\bsbeta)\eta_t^{-1}$ as before.

Our results for the Boltzmann case are less detailed than for the Fokker-Planck case. Although it
is not know whether \eqref{betranc} admits a stationary solution we can demonstrate that
there is no stationary solution of exponential type.
\begin{theorem} \label{Boltz-thm}
Equation~\eqref{betranc} admits a global solutions if $G_0 \in L^1$, which preserve mass and
the renormalized Cauchy stress tensor
$$ T_{\eta_t} =\frac{1}{2} \int_{\R^\di} \bsp \otimes \bsp \, G\, \dd \bsp. $$
The collision invariants of $Q_{\eta_t}$ are $1$, $\eta_t^{-1}\bsp$ and $|\eta_t^{-1}\bsp|^2$. The solutions
$G_t$ do not converge to a function of exponential form $K\exp(h(\bsp))$  with $ h(\alpha \bsp) =\alpha^r h(\bsp)\, \, \forall \alpha>0, \forall \bsp \in \R^\di$ and a fixed $r>0$
as $t \to \infty$. There exists a $G_\infty \in L^1$ with $\|G_\infty\|_{L^1}=1$ and a sequence $t_j \to \infty$ as $j \to \infty$ such that $G_{t_j}$ converges weakly in $L^1$ to $G_\infty$.
\end{theorem}

The proof involves several parts. The collision invariants are determined in subsection \ref{sub41}, the rescaling $\eta_t$ is determined in subsection \ref{sseta}. The shape of universal equilibria are discussed in subsection \ref{ssnoMax}. The global bounds leading to tightness are given in subsection \ref{ss44}, which completes the proof.

\begin{remark}
The question whether \eqref{betranc} admits a Lyapunov functional appears to be open.
It is not hard to see, if
$$ S[G] = \int_{\R^\di} G\, \log G\, \dd \bsp,$$
that we have
\begin{equation} \label{ent-prod}
 \frac{\dd S}{\dd t} \leq -\frac{1}{4} \int_{\R^{4}}\int_{S^{1}}  \min_{s \in [0,1]}\frac{(GG'-G_*G_*')^2}{s\,GG' +(1-s)G_*G_*'}[\nu \cdot \eta_t^{-1}(\bsp - \bsp') ]_+\,\dd \nu\,\dd \bsp' \dd \bsp  -\tr \A
\end{equation}
However the behaviour of $\A$ , which will be linked to the stress rates $P$ in~\eqref{etaode} below, cannot be determined. Note that the first term in~\eqref{ent-prod} is analogous to the standard entropy production in the case where $\mu=0$. In particular, it is non-positive. However $\tr P$ is not necessarily negative and in contrast to Remark \ref{FP-ent} we cannot conclude that $S$ is Lyapunov functional.
\end{remark}

\begin{proof}
\begin{eqnarray*}
\frac{\dd S}{\dd t} &=& \int_{\R^\di} (1+\log G) \, \partial_t G\, \dd \bsp\\
&=& \int_{\R^\di} (1+\log G)(Q_{\eta_t} -\nabla\cdot(G \A\bsp))\, \dd \bsp = \int_{\R^\di} (Q_{\eta_t} -
\nabla\cdot(G \A\bsp))\, \log G\, \dd \bsp\\
&=&  \int_{\R^\di} Q_{\eta_t} \log G\,\dd \bsp + \int_{\R^\di} \frac{1}{G} \nabla_{\bsp} G \cdot \,G \A \bsp\, \dd \bsp\\
&=& \int_{\R^{4}}\int_{S^{d-1}} (G_* G_*' -G G')\,\log G\,
[\nu \cdot \eta_t^{-1}(\bsp - \bsp') ]_+\,\dd \nu\,\dd \bsp' \dd \bsp  -\tr \A\\
&=&\frac{1}{4} \int_{\R^{4}}\int_{S^{d-1}} (G_* G_*' -G G')\,(\log G\,+\log G' - \log G_* -\log G_*')\,
 [\nu \cdot \eta_t^{-1}(\bsp - \bsp') ]_+\,\dd \nu\,\dd \bsp' \dd \bsp  -\tr \A\\
&=& \frac{1}{4} \int_{\R^{4}}\int_{S^{d-1}} (G_* G_*' -G G')\,(\log GG' - \log G_*G_*')\,
 [\nu \cdot \eta_t^{-1}(\bsp - \bsp') ]_+\,\dd \nu\,\dd \bsp' \dd \bsp  -\tr \A\\
&\leq & -\frac{1}{4} \int_{\R^{4}}\int_{S^{d-1}}  \min_{s \in [0,1]}\frac{(GG'-G_*G_*')^2}{s\,GG' +(1-s)G_*G_*'}[\nu \cdot \eta_t^{-1}(\bsp - \bsp') ]_+\,\dd \nu\,\dd \bsp' \dd \bsp  -\tr \A,
\end{eqnarray*}
as required.
%he claim now follows from~\eqref{etaode} which implies that $\tr \,\A = - \tr P$.
\end{proof}

\subsection{Collision invariants and stationary solutions}\label{sub41}
If we ignore $\tr \A$ in \eqref{ent-prod} it is well known that the numerator vanishes if $G$ depends only on collision invariants $k_{\eta_t}(\bsp)$ which are characterized by
$$ k+k' =k_*+k_*'.$$
We determine the collision invariants below, but $|\bsp|^2$ is not a collision invariant for general $\eta_t$.
Note that every collision invariant $k$ generates a stationary solution $G = \exp(k)$ for $Q_{\eta_t}$, but not in general for the full equation~\eqref{betranc}.

\begin{lemma}\label{lem:ci}
If $G$ is a zero of $Q_{\eta_t}$, i.e. $Q_{\eta_t}[G]=0$, then there exists $a,c\in \R$, $\bsb \in \R^\di$ such that
$$ G(\bsp) = \exp\(a+ \bsb \cdot \eta_t^{-1} \bsp + c\,|\eta_t^{-1} \bsp|^2\).$$
Let $k_{\eta_t}(\bsp) = \bsp \cdot K_{\eta_t}\bsp$ with $K_{\eta_t} = (\eta_t^*)^{-1} \eta_t^{-1}
\in \R^{\di\times \di}_{\mathrm{sym}}$.
Then $k$ is a quadratic collision invariant.
\end{lemma}
\begin{proof}
Recall that
\begin{equation} \label{Qtransf}
Q_{\eta_t}[G](\bsp) = Q[g](\eta_t^{-1} \bsp),
\end{equation}
where $G(\bsp)= \det\eta_t^{-1} g(\eta_t^{-1} \bsp)$.
It is well known (e.g. \cite{CIP94}, Sec. 3.2) that $Q[g] = 0$ if and only if $g = \exp(k(\bsp))$
where $k$ is a collision invariant, i.e.
$$ k(\bsp) = a+\bsb\cdot \bsw +c|\bsw|^2.$$
Thus, $Q_{\eta_t}[G]=0$ if and only if $G = a+\bsb\cdot \eta_t^{-1}\bsp + c|\eta_t^{-1}\bsp|^2$, which is the claim.
\end{proof}

\subsection{Choice of rescaling}\label{sseta}
It is not hard to see that collisions do not conserve
standard kinetic energy of $\bsp$ and $\bsp'$. Indeed
\begin{eqnarray} \nonumber
|\bsp_*|^2+|\bsp_*'|^2 & =&
|\bsp -\eta_t\nu \otimes \nu \eta_t^{-1}(\bsp-\bsp')|^2 +
|\bsp' +\eta_t\nu \otimes \nu \eta_t^{-1}(\bsp-\bsp')|^2\\
\nonumber
&=&|\bsp|^2+|\bsp'|^2-2(\bsp-\bsp')\cdot \eta_t\nu \otimes \nu
\eta_t^{-1}(\bsp-\bsp')+2(\bsp-\bsp')\cdot\eta_t^{-1} \nu \otimes \nu
\eta_t^2 \nu \otimes \nu \eta_t^{-1} (\bsp -\bsp')\\
&=& |\bsp|^2 + |\bsp'|^2 + (\bsp-\bsp')\cdot C_\nu (\bsp-\bsp'), \label{endef}
\end{eqnarray}
where
$$ C_\nu = [(\nu \cdot \eta_t^2\nu) \eta_t^{-1}\nu \otimes \nu \eta_t^{-1}
-\eta_t\nu \otimes \nu \eta_t^{-1}]_\mathrm{sym}.$$
% = [B_\nu^*B_\nu - B_\nu^2]_\mathrm{sym}=[B_\nu^*B_\nu - B_\nu]_\mathrm{sym} = B_\nu^*B_\nu - \frac12(B_\nu+B_\nu^*),$$
%with $B_\nu=\eta_t \nu \otimes \nu \eta_t^{-1}$.

%The effective energy defect associated with the collision of
%2 particles with velocities $\bsp$ and $\bsp'$ is given by
%$$ \bar C_\eta_t(\bsp-\bsp') = \int_{-\frac{\pi}{2}}^{\frac{\pi}{2}}\cos\theta\,C_\nu\,
% \dd \theta,$$
%where $\nu=(\cos(\theta_0+\theta), \sin(\theta_0+\theta))$,
%$$ \theta_0 = \arctan\(\frac{\eta_t^{-1}(\bsp-\bsp')\cdot e_2}{\eta_t^{-1}(\bsp-\bsp')\cdot e_1}\).$$

%so that $\theta = \tr \, T$.
The stress rates are given by
$$ P_{\eta_t} = \frac{1}{2}\int_{\R^\di} \bsp\otimes \bsp \, Q_{\eta_t} \, \dd \bsp,$$
so that
\begin{eqnarray}
\nonumber
\tr P_{\eta_t} &=& \frac{1}{2} \int_{\R^{4}} \int_{S^{1}}|\bsp|^2\, (G_* G_*' -G G')\,[\nu \cdot \eta_t^{-1}(p-p')]_+ \dd \nu\,\dd \bsp'\, \dd \bsp \\
\nonumber
&=& \frac{1}{4} \int_{\R^{4}} \int_{S^{1}}(|\bsp|^2+|\bsp'|^2)\, (G_* G_*' -G G')\,[\nu \cdot \eta_t^{-1}(p-p')]_+ \dd \nu\,\dd \bsp'\, \dd \bsp\\
\nonumber
&=& \frac{1}{4} \int_{\R^{4}} \int_{S^{1}}(|\bsp_*|^2+|\bsp_*'|^2)\, G G'\,
[\nu \cdot \eta_t^{-1}(p-p')]_+ \dd \nu\,\dd \bsp_*'\, \dd \bsp_*\\
\nonumber
&&  - \frac{1}{4} \int_{\R^{4}} \int_{S^{1}}(|\bsp|^2+|\bsp'|^2) \, G G'\,[\nu \cdot \eta_t^{-1}(p-p')]_+ \dd \nu\,\dd \bsp'\, \dd \bsp\\
\label{trP}
&=& \frac{1}{4}\int_{\R^{4}} \int_{S^{1}}(\bsp - \bsp') \cdot C_\nu \,(\bsp-\bsp') \, G G' \,
[\nu \cdot \eta_t^{-1}(p-p')]_+ \dd \nu\,\dd \bsp'\, \dd \bsp.
\end{eqnarray}
The first equation holds because the order of integration can be exchanged, the last equation follows from~\eqref{endef}.
In particular one finds that
\begin{equation}\label{eq:PId}
 \tr P_\Id = 0.
\end{equation}
Our aim is to construct a time-dependent transformation $\eta_t\in
\R^{d\times d}_\mathrm{sym}$ such that the renormalized
Cauchy stress tensor
$$ T_{\eta_t} =\frac{1}{2} \int_{\R^\di} \bsp \otimes \bsp \, G(\bsp)\, \dd \bsp $$
is constant. We have already seen in section~\ref{sec:FP} that $T_{\eta_t} =  \Id$ if
$$ \eta_t^{-2} = \frac{1}{2}\int \bsw \otimes \bsw \,g_t(\bsw)\, \dd \bsw,$$
and $g_t$ is a solution of \eqref{be}.
%We consider the evolution of the kinetic energy
%$$\theta_\eta_t(t)= \frac{1}{2}\int_{\R^2} |\bsp|^2 \, G(\bsp)\, \dd \bsp.$$
Differentiating the Cauchy stress with respect to $t$ and using $T_{\eta_t} = \frac{1}{2} \Id$ gives
\begin{eqnarray*}
\frac{\dd T_{\eta_t}}{\dd t} &=& \frac{1}{2}\int_{\R^\di} \bsp \otimes \bsp \,(Q_{\eta_t} -\nabla\cdot(G \A \bsp))\, \dd \bsp
= %P + \int_{\R^2} \bsp \otimes \A\bsp \, G(\bsp)\,\dd \bsp=
P+ \frac{1}{2}(\A+\A^*) = P+\A_\mathrm{sym}.
\end{eqnarray*}
Thus, we have obtained a non-autonomous system of ordinary differential equations
\begin{equation} \label{etaode}
P + \A_\mathrm{sym}=0.
\end{equation}
Obviously \eqref{etaode} is the analogue of equations \eqref{resmeq2}. % and \eqref{resmeq3}.

\subsection{Are there stationary solutions that are of exponential form?}\label{ssnoMax}
Assume that $G$ is of the exponential form , i.e. there exist $h \in C^1(\R^\di\setminus\{0\})$ and homogeneity exponent $r>0$
$$ G^{(h)}(\bsp) = K\exp(h(\bsp)), \quad \mbox{ with } h(\lambda \bsp)=\lambda^r h(\bsp) \quad  \forall \lambda >0, \forall \bsp \in \R^\di.$$
Note that by integrability of $G$ this implies $h <0$ on $\R^\di$. Now $G$ is substituted into the differential equation \eqref{betranc}. Note first that
\begin{equation}\label{eq:maxrhs}
\nabla_\bsp \cdot (G^{(h)}(\bsp) \cdot \A\bsp) = (\tr \A + \nabla h \cdot \A\bsp)\, G^{(h)}(\bsp)
%=(\tr(\eta_t^{-1} \dot{\eta_t} + \nabla h \cdot \A\bsp)G(\bsp)
=k_t(\bsp)\, G(\bsp)
\end{equation}
where $k$ is the sum of a constant and homogeneous function of order $r$ in $\bsp$.

Furthermore in the collision term we denote
$$ \bsp_*=\bsp- \eta_t \nu \otimes \nu \eta_t^{-1} (\bsp-\bsp'), \quad
\bsp_*'=\bsp'+\eta_t \nu \otimes \nu \eta_t^{-1}(\bsp-\bsp'),$$
then it has the form
\begin{eqnarray*}
&& Q_{\eta_t}[G](\bsp)\\&=& K^2\int_{\R^\di} \int_{S^{1}}\Bigl\{ \exp\(h(\bsp_*) + h(\bsp_*')\)-\exp\( h(\bsp)+h(\bsp')\)
\Bigr\}[\nu\cdot \eta_t^{-1}(\bsp-\bsp') ]_+\, \dd \nu \, \dd \bsp' \nonumber\\
&=&K^2e^{h(\bsp)} \,  \int_{\R^\di}\int_{S^{1}} \left[\exp\(h(\bsp_*) + h(\bsp_*')-h(\bsp)-h(\bsp')\) -1\right]
\exp\(h(\bsp')\)\,[\nu\cdot \eta_t^{-1}(\bsp-\bsp') ]_+\, \dd \nu \, \dd \bsp'\nonumber\\
&=&K^2e^{h(\bsp)} \,  \int_{\R^\di} e^{h(\bsp-\bsq)}\int_{S^{1}} \nonumber\\&& \underbrace{\left(
\exp\(h(\bsp- \eta_t \nu \otimes \nu \eta_t^{-1} \bsq) + h(\bsp-\bsq+\eta_t \nu \otimes \nu \eta_t^{-1}\bsq)-h(\bsp)-h(\bsp-\bsq)\)-1\right)\,[\nu\cdot \eta_t^{-1} \bsq]_+}_{=j(\bsq)}\, \dd \nu \, \dd \bsq.
\end{eqnarray*}
To cancel the expression in \eqref{eq:maxrhs}, the integrals over $j$ are necessarily $O(|\bsp|^r)$ as $|\bsp| \to \infty$. First we consider the case when $\bsq=\bsp$.
Then  the  exponent can be simplified to
$$i(\nu):=h(\bsp -\eta_t \nu \otimes \nu \eta_t^{-1} \bsp) + h(\eta_t \nu \otimes \nu \eta_t^{-1}\bsp)-h(\bsp) -h(0)$$
If there are $\nu$ and  $\bsp$ such that
\begin{equation}\label{eq:hsumpos}
i(\nu)>0,
\end{equation}
then $j(\bsq)$ grows exponentially for a sector of $\bsq$ in $\R^\di$. Due to continuous dependence of this sector on $\nu$,   the integral
$\int_{S^{1}}  j(\bsq)\, \dd \nu$ still grows exponentially on some sector in $\R^\di$. By choosing $\bsp$ in such a sector, we obtain constants $c_1,c_2$ such that
$$ Q_{\eta_t}[G](\bsp) \geq (c_1 j(\bsp) -c_2)G(\bsp),$$
which cannot equal to a term $k_t(\bsp) G(\bsp)$ as in \eqref{eq:maxrhs}, where $k$ is bounded by a polynomial.

We now establish the existence of some $\nu$ such that \eqref{eq:hsumpos} holds. For $t=0$ we have $\eta_t=\Id$ and by \eqref{eq:PId} and \eqref{etaode} $\tr \A =0$, then  equations~\eqref{betranc} and \eqref{eq:maxrhs} imply for $\bsp =0$ that
\begin{align*} %\tr \eta_t^{-1}\dot \eta_t
0&=Q_{\eta_t}[G^{(h)}]=Q[G^{(h)}]\\
&=K^2 \,  \int_{\R^\di} e^{h(-\bsq)} \int_{S^{1}}\left(
\exp\(h(- \nu \otimes \nu \bsq) + h(-\bsq+\nu \otimes \nu \bsq)-h(-\bsq)\)-1\right)\,[\nu\cdot \bsq]_+\, \dd \nu \, \dd \bsq.
\end{align*}
Unless $r=2$, when $h$ is a collision invariant, this implies that the exponent attains positive and negative values. By exchanging the roles of $\bsp$ and $\bsp'$, we hence obtain
that \eqref{eq:hsumpos} will hold for some $\bsp$ and $\nu$, hence ruling out any $h$ with homogeneity exponent $r>0$, $r \neq 2$.

Now consider the only remaining case $h(\bsp)= -d |\eta_t^{-1} \bsp|^2$ for some $d>0$.
The collision invariance of $h$  implies $Q_{\eta_t}[G^{(h)}]\equiv 0$. Furthermore $\eta_t=\Id$ is constant as $G^{(h)}$ is constant, i.e. $\partial_t G^{(h)}=0$.
Hence we obtain the equation \[\nabla_\bsp \cdot (G^{(h)}(\bsp) \cdot \A\bsp) =(\tr \A +2 d \bsp \cdot F \bsp)G^{(h)} =0.\] For $\eta_t=\Id$, then  $\A= -\mu \bsalpha \otimes \bsbeta$, which is incompatible with   $\tr \A -2 d \bsp \cdot \A \bsp=-2 d \bsp \cdot \A \bsp=0$, thus ruling out the collision invariant.

\subsection{Completing the proof of Theorem \ref{Boltz-thm}}\label{ss44}
The regularity of the solution follows from the regularity proposition below.

\begin{proposition} \label{prop:regB} If $\int_{\R^\di} g_0(\bsw) (1+|\bsw|^2)\, \dd \bsw < \infty$,
then \eqref{be} admits a unique mild solution for all $t>0$.
\end{proposition}

\begin{proof}
The transport term $\mu\, v\,\partial_u g$ is the generator of the strongly continuous semigroup
$X_t: L^1_2(\R^\di) \to L^1_2(\R^\di)$ on the space of integrable function that satisfy
$\int_{\R^\di} g_0(\bsw) (1+|\bsw|^2)\, \dd \bsw < \infty$. The semigroup is given explicitly by
$(X_tg)(\bsw) =g((\Id + \mu\,t\,\bsalpha \otimes \bsbeta) \bsw)$. Furthermore by Povzner's inequality $Q$ is a continuous nonlinear operator
on $L^1_2(\R^\di)$. This gives the existence of the unique mild solution to \eqref{be} given by
\begin{equation}\label{eq:mild}
 g_t = X_t g_0 + \int_0^t X_{t-s} Q(g_s,g_s)\dd s.
\end{equation}
We show that we can continue this solution globally by showing that  $\int_{\R^\di} g_t(\bsw) (1+|\bsw|^2)\, \dd \bsw < \infty$ for all times.
First recall that
\begin{equation} \label{colinv}
\int_{\R^\di} |\bsw|^2 \, Q[g]\, \dd \bsw = 0,
\end{equation}
hold for any density $g$ with $\theta[g] = \frac{1}{2}\int |\bsw|^2\, \dd \bsw <\infty$ because $|\bsw|^2$ is a collision invariant.
Differentiating $\theta$ with respect to $t$ gives
\begin{eqnarray*}
\frac{\dd \theta}{\dd t} &=& \frac{1}{2}\int_{\R^\di} |\bsw|^2( \mu\,\nabla\cdot(g \,\bsalpha \otimes \bsbeta \bsw)+ Q[g])\, \dd \bsw\\
&=&  \frac{1}{2} \underbrace{\int_{\R^\di} |\bsw|^2 \, Q[g]\, \dd \bsw}_{=0 \text{ by \eqref{colinv}}} -\mu \int_{\R^\di} g\, \bsw \cdot \bsalpha \otimes \bsbeta \bsw\,  \dd \bsw\\
&\leq& \frac{\mu}{2}\int_{\R^\di} |\bsw|^2\, g\, \dd \bsw = |\mu| \, \theta[g].
\end{eqnarray*}
Thus, $\theta[g] \leq e^{|\mu| t}\, \theta[g_0]$. By a similar argument, $\int_{\R^\di} g_t(\bsw) \, \dd \bsw =\int_{\R^\di} g_0(\bsw) \, \dd \bsw $,
so $\int_{\R^\di} g_t(\bsw) (1+|\bsw|^2)\, \dd \bsw$ remains bounded for bounded times, such that \eqref{eq:mild} defines global mild solution, which is unique by a Gronwall argument.
\end{proof}

Then  we  transform the mild solution $g$  as in the paragraph preceding \eqref{betranc} to obtain a global solution $G$  of \eqref{betranc}.

The choice of $\eta_t$ in subsection \ref{sseta} give the preservation of  mass and energy for $G_t$,
this immediately gives the weak convergence of subsequences to some limit points. The collision invariants are characterised in lemma \ref{lem:ci}. The shape of possible equilibrium
is analysed in subsection \ref{ssnoMax}.

\section{Conclusion}
We studied two closely related equations in kinetic theory, the Fokker-Planck equation and the Boltzmann equation
with shear boundary conditions. The boundary conditions are not compatible with the conservation of energy. After rescaling the velocities in an anisotropic fashion we obtain renormalized equations which have the property that
solutions conserve all second moments, and in particular the energy. The renormalized Fokker-Planck equation admits Maxwellian equilibria and the long-time behaviour of renormalized solutions can be characterized completely. More precisely, we show rigorously that as $t$ tends to infinity solutions converge at an algebraic rate to the Maxwellian with the appropriate second moments.

On the other hand the renormalized Boltzmann equation does not admit equilibria of exponential type including Maxwellians. Indeed, due to the non-autonomous nature of the shape equation \eqref{betranc} there might be no equilibria at all. We conjecture that for large time solutions of the renormalized Boltzmann converge to a limiting density, but a rigorous proof is not available.

Results on the existence of self-similar profiles (i.e. equilibria for the shape equation) and long-time behaviour in the case of soft interaction
potentials
have been obtained in \cite{JNV17} and \cite{JNV18}. In \cite{JNV17} the existence of stationary self-similar solutions is established rigorously for Maxwellian molecules (where the repulsive force between particles at distance $r$ is $r^{-5}$) after isotropic rescaling (where $\eta$ is a multiple of the identity). Detailed information about energy flux can then be derived.
% converge to a Maxwellian.
Moreover, in \cite{JNV18} formal calculations covering the supercritical case where the force decays faster than $r^{-5}$ are being presented. Based on these calculations the authors conjecture that after isotropic rescaling in the supercritical case solutions converge to a Maxwellian.

It is noteworthy that the analysis in \cite{JNV17} and \cite{JNV18} also covers other objectivity conditions than the
simple shear, for example homogeneous dilations where $S=(\Id, -\Id)$, and other choices. In view of these results it will be worthwhile to extend our approach to other objectivity conditions and explore different choices for the collision operator in the Boltzmann equation.

%\\[1em]
%comment on lower estimates for all initial data in FP case

%Hier bin ich mir nicht sicher. Ich verstehe die Bedeutung deines Satzes nach Gleichung (33) nicht. Bitte erinnere mich unter welchen Umstaenden die Loesungen schneller
%als algebraisch konvergieren koennten falls die sechsten Momente existieren.

%discussion of existence of limit object for Boltzmann case. Da wir auf die Arbeiten von James et al verweisen, ist das doch erledigt.

%%% \subsection{Are there Maxwellian stationary solutions?}
%%% Assume that $G$ is a unnormalized Maxwellian, i.e.
%%% $$ G(p) = \exp(-\beta \, |\bsp|^2).$$
%%% Then
%%% \begin{eqnarray*}
%%% Q_\eta_t[G,G](p) &=& \int_{S^{1}}\int_{\R^2} \Bigl\{ \exp\(-\beta|\(\bsp- \eta_t \nu \otimes \nu \eta_t^{-1} (\bsp-\bsp')|^2 + |\bsp'+\eta_t \nu \otimes \nu \eta_t^{-1}(\bsp-\bsp')|^2\)\)\\
%%% &&-\exp\( -\beta \(|\bsp|^2+|\bsp'|^2\)\)
%%% \Bigr\}[(\bsp-\bsp')\cdot \eta_t^{-1} \nu]_+\, \dd \nu \, \dd \bsp'\\
%%% &=&e^{-\beta |\bsp|^2} \, \int_{S^1} \int_{\R^2}\left[e^{-\beta(\bsp-\bsp')\cdot C_\nu
%%% (\bsp-\bsp')}-1\right] e^{-\beta |\bsp'|^2}\,[(\bsp-\bsp')\cdot \eta_t^{-1} \nu]_+\, \dd \nu \, \dd \bsp'\\
%%% &=&e^{-\beta |\bsp|^2} \, \int_{S^1} \int_{\R^2} e^{-\beta |\bsq+\bsp|^2}\underbrace{\left(e^{-\beta\bsq\cdot C_\nu
%%% \bsq }-1\right)\,[\bsq\cdot \eta_t^{-1} \nu]_+}_{=h(\bsq)}\, \dd \nu \, \dd \bsq
%%% \end{eqnarray*}
\section{Appendix: Proofs of Propositions~\ref{fpcons} and Remark~\ref{FP-ent} }
\begin{proof}[Proof of Proposition~\ref{fpcons}]
Assume that $h(\bsw)= a + \bsb \cdot \bsw$ is affine. Observe that \eqref{shear} implies for each $\bsw$ that
\begin{eqnarray*}
&&\nabla_\bsz f(x\,    \bsalpha   ,\bsw)\cdot\bsalpha =0\\
&& \nabla_\bsz f \cdot \bsbeta = - \mu \nabla_{\bsw} f\cdot \bsalpha
\end{eqnarray*}
Consider now the quantity
\begin{eqnarray*}
H&=&  \int_{\R^\di} h(\bsw)\, \partial_t f(x\,   \bsalpha   , \bsw)\,\dd \bsw\\
&=& \int_{\R^\di} h(\bsw)\, \(\Delta_\bsw f + \rho\,\theta^{-1}\,\nabla_\bsw\cdot(f(x\,   \bsalpha   ,\bsw)\,(\bsw-\rho^{-1}\bar\bsv))-\nabla_\bsz f(x\,   \bsalpha   ,\bsw)\cdot\bsw\)\,\dd \bsw\\ %\, \dd x\\
&=& \int_{\R^\di}\left\{\(\Delta_\bsw h -\rho\,\theta^{-1}\nabla h\cdot (\bsw-\rho^{-1} \bar\bsv)\)\, f(x\,   \bsalpha   ,\bsw) + \mu \,(\nabla_{\bsw} f \cdot \bsalpha) \, (\bsw\cdot \bsbeta) \right\}\,\dd \bsw%\,\dd x.
\end{eqnarray*}
Clearly $\Delta h=0$ as $h$ is affine. Moreover $\int_{\R^\di} \bsb \cdot (\bsw-\rho^{-1}   \bar\bsv  )\, f(\bsw)\,\dd \bsw=0$ by the definition of    $\bar\bsv$    and finally
$ \int_{\R^\di}(\nabla_\bsw f(x\,   \bsalpha   ,\bsw)\cdot \bsalpha)\,(\bsw\cdot \bsbeta)\,\dd \bsw= 0$ by partial integration.
This implies that
$$ t \mapsto \int_{\R^\di} h(\bsw) \, f_t(x\,\bsa,\bsw)\, \dd \bsw  \text{ is  constant},$$
for all $x\in \R$ and by  \eqref{shear} this is constant and thereby the first claim.

Next, assume that $\mu =0$ and $h(\bsw) = \frac{1}{2}|\bsw|^2$. Repeating the previous calculation we obtain
\begin{eqnarray*}
H&=& \int_{\R^\di} h(\bsw)\, \partial_t f(x\,   \bsalpha   , \bsw)\,\dd \bsw \\
&=&\int_{\R^\di} (2- \theta^{-1}|\bsw|^2)\,f(x\,   \bsalpha   , \bsw)\,\dd \bsw = \int_0^1 \(2\,\rho(x\,\alpha)-2\,\theta^{-1}\rho\,\theta(x\,\alpha)\)\,\dd \bsw=0.
\end{eqnarray*}
Finally we demonstrate that $f^M$ is a stationary solution. One finds that
\[%begin{align*}
\rho = -\frac{\pi}{c}\, \exp\(a- \frac{|\bsb|^2}{4c}\), \quad
\bsw_0 = \frac{\pi}{2 c^2}\,{\exp\(a- \frac{|\bsb|^2}{4c}\)} \bsb,\quad
\theta = \frac{\pi}{2 c^2}\,{\exp\(a- \frac{|\bsb|^2}{4c}\)},
\]%\end{align*}
in particular $\rho\, \theta^{-1} = -2c$ and $\rho^{-1} \bsw_0=-\frac{1}{2c}\bsb$. Then
\begin{eqnarray*}
L f^M&=&\Delta_\bsw f^M + \rho\,\theta^{-1}\,\nabla_\bsw\cdot(f^M\,(\bsw-\rho^{-1}\bsw_0))-\nabla_\bsz f^M\cdot\bsw\\
&=& \(|\nabla h|^2+\Delta h-4c-2c\,\nabla h(\bsw)\cdot \(\bsw+\frac{1}{2c}\bsb\)\)\,f^M\\
&=& \(4c^2 \,|\bsw|^2+ |\bsb|^2+4c \,\bsb \cdot \bsw+4c- 4c -2c\, (2c\,\bsw+\bsb)\cdot\(\bsw +\frac{1}{2c}\bsb\)\)\,f^M\\
&=& \(\(4c^2 -4c^2\) |\bsw|^2 + \(4c\,\bsb-2c\,(\bsb+\bsb)\)\cdot\bsw
+|\bsb|^2-|\bsb|^2\)\,f^M=0.
\end{eqnarray*}

To prove convergence for general spatially homogeneous initial datum $f_0$, we rewrite the equation in a spirit similar to subsection~\ref{sec:shape} and equation~\eqref{eq:u}. Using that mass $\rho$, momentum $\bsw_0$ and energy $\theta$
remain constant along solutions, choose $f^M$ such that its triple  $\rho, \bsw_0, \theta$ coincide with the of $f_0$. Then write $f=u f^M$ with $u \in L^2(\R^\di, \dd f^M)$. To show $L^1$ convergence it is enough to show that $u \to 0$ in $L^2(\R^\di, \dd f^M)/ \mbox{span}(1)$ by H\"{o}lder's inequality. The relative profile $u$
satisfies the equation
\begin{equation*} %\label{eq:ucoer}
  \partial_t u =\Delta u + (\bsb+2c \bsw) \cdot \nabla u= -A^*A u,
\end{equation*}
where $A=\nabla u$ and $A^* .= - \nabla . - (\bsb+2c \bsw) .$ is its adjoint operator in $L^2(\R^\di, \dd f^M)/ \mbox{span}(1)$ with inner product $ \langle .,.\rangle$. Then
 \begin{equation*}%\label{eq:ucoer}
  \partial_t \langle u, u \rangle =2\langle \partial_t u, u \rangle = -2 \langle A u,Au \rangle  \leq -2 C \langle  u,u \rangle
\end{equation*}
using a Poincar\'{e} inequality as in \cite[A.19]{VillHyp}, which then gives the required exponential convergence of $u$ to $0$ in $L^2(\R^\di, \dd f^M)/ \mbox{span}(1)$.

\end{proof}

\begin{proof}[Proof of Remark~\ref{FP-ent}]
We calculate $\frac{\dd}{\dd t} S[G]$ along solutions of \eqref{eq:shape}, noticing that $G$ is smooth with respect to $\bsp$ as $g$ is smooth for $t>0$, so we can perform integration by parts etc.
\begin{align*}
   &\frac{\dd}{\dd t} S[G_t]\\=&\frac{\dd}{\dd t} \int_{\R^\di} G_t(\bsp) \ln \frac{G_t(\bsp)}{\exp(-|\bsp|^2/2)} \,\dd \bsp=
   \int_{\R^\di} \(\ln \frac{G_t(\bsp)}{\exp(-|\bsp|^2/2)}+1\) \partial_t G_t(\bsp) \dd \bsp\\
   \stackrel{\eqref{eq:shape}}{=}&\int_{\R^\di} \(\ln \frac{G_t(\bsp)}{\exp(-|\bsp|^2/2)}+1\) \nabla_\bsp \cdot\(G_t(\bsp)\( \theta^{-1} \Id - \A\)\bsp + T^{-1} \nabla G_t(\bsp)\) \dd \bsp\\
    \stackrel{\mbox{ibp}}{=}&-\int_{\R^\di} \nabla \(\ln \frac{G_t(\bsp)}{\exp(-|\bsp|^2/2)}+1\) \cdot \left(G_t(\bsp)\( \theta^{-1}\Id - \A\)\bsp + T^{-1} \nabla G_t(\bsp)\right)\, \dd \bsp\\
    =&-\int_{\R^\di} \frac{1}{G_t(\bsp)}\left( \nabla G_t(\bsp) + G_t(\bsp) \bsp \right)\cdot \left(G_t(\bsp)\(\theta^{-1} \Id - \A\)\bsp + T^{-1} \nabla G_t(\bsp)\right) \dd \bsp
\end{align*}
Next we split $\A$ into its symmetric  and its anti-symmetric part, which are given by $\frac 1 2 (\A+\A^*)$ and $\frac 1 2 (\A-\A^*)$  respectively. For the symmetric part we use \eqref{resmeq2} and
find that
\begin{align*}
&-\int_{\R^\di} \frac{1}{G_t(\bsp)}\left( \nabla G_t(\bsp) + G_t(\bsp) \bsp \right)\cdot \left(G_t(\bsp)\(\theta^{-1} \Id - \A\)\bsp + T^{-1} \nabla G_t(\bsp)\right)\, \dd \bsp\\
    =&-\int_{\R^\di} \frac{1}{G_t(\bsp)}\left( \nabla G_t(\bsp)+ G_t(\bsp)\bsp \right)\cdot T^{-1} \left(G_t(\bsp) \bsp +  \nabla G_t(\bsp)\right) \,\dd \bsp\\
    &+\frac{1}{2}\int_{\R^\di} \frac{1}{G_t(\bsp)}\left( \nabla G_t(\bsp) +  G_t(\bsp) \bsp \right)\cdot G_t(\bsp)\,(\A-\A^*)\bsp\, \dd \bsp\\
    %=&-\int_{\R^2} \frac{1}{G(\bsp)}\left|\eta_t\left(G(\bsp) \bsp +  \nabla G(\bsp)\right)\right|^2 \dd \bsp\\
    %&-\int_{\R^2} \nabla G(\bsp) \cdot (\A-\A^*)\bsp \, \dd \bsp-\int_{\R^2}\bsp\cdot (\A-\A^*)\bsp \, \,G(\bsp)\,\dd \bsp \\
    =&-\int_{\R^\di} \frac{1}{G_t(\bsp)}\left|\eta_t\left(G_t(\bsp) \bsp +  \nabla G_t(\bsp)\right)\right|^2\, \dd \bsp+ \frac{1}{2}\int_{\R^\di}\nabla G(\bsp) \cdot (\A-\A^*)\bsp\,\dd \bsp\\
    &\quad +\frac{1}{2}\int_{\R^\di} \bsp \cdot (\A-\A^*)\bsp\, G_t(\bsp)\,  \dd \bsp \\
    =&-\int_{\R^\di} \frac{1}{G_t(\bsp)}\left|\eta_t\left(G_t(\bsp) \bsp +  \nabla G_t(\bsp)\right)\right|^2 \,\dd \bsp,
\end{align*}
the other two integrals are zero, the middle one by integration by parts and the final one due to the anti-symmetry of $\A-\A^*$.

Hence $S[G_t]$ decays unless
$G_t(\bsp) \bsp +  \nabla G_t(\bsp)=0$, the only differentiable solution in $L^1(\R^\di)$ are multiples of the Maxwellian $G^M$.
\end{proof}

\bibliographystyle{plain}
\bibliography{omd}

\end{document}